\documentclass[brochure,english,12pt]{smfbourbaki}

\usepackage[T1]{fontenc}
\usepackage{lmodern,amssymb,amsmath,bm}

\DeclareFontFamily{U}{mathx}{}
\DeclareFontShape{U}{mathx}{m}{n}{<-> mathx10}{}
\DeclareSymbolFont{mathx}{U}{mathx}{m}{n}
\DeclareMathAccent{\widehat}{0}{mathx}{"70}
\DeclareMathAccent{\widecheck}{0}{mathx}{"71}

\DeclareMathOperator\supp{supp}
\DeclareMathOperator\Bohr{Bohr}
\DeclareMathOperator\Spec{Spec}
\DeclareMathOperator\rk{rk}
\usepackage{babel}

\usepackage[utf8]{inputenc}

\usepackage[colorlinks=true, linkcolor=blue, citecolor=red,
urlcolor=blue]{hyperref}
\usepackage{microtype} %%% may be commented until final version

\usepackage[
backend=biber,
style=authoryear, 
citestyle=authoryear-comp,
maxnames=7
]{biblatex}
\usepackage{csquotes}

\DeclareDelimFormat{nameyeardelim}{\addcomma\space}

\DeclareNameAlias{sortname}{given-family}

\renewcommand{\bibnamedash}{\leavevmode\raise3pt\hbox to3em{\hrulefill}\space}

\AtEveryBibitem{%
  \clearfield{issn} % Remove issn
  \clearfield{isbn} % Remove isbn
  \clearfield{doi} % Remove doi

  \ifentrytype{online}{}{% Remove url except for @online
    \clearfield{url}
  }
}

\renewbibmacro{in:}{%
    \ifentrytype{article}{}{\printtext{\bibstring{in}\intitlepunct}}}

\DeclareFieldFormat[article,periodical,inreference]{number}{\mkbibparens{#1}}
\DeclareFieldFormat[article,periodical,inreference]{volume}{\mkbibbold{#1}}
\renewbibmacro*{volume+number+eid}{%
    \printfield{volume}%
    \setunit*{\addthinspace}% NEW
    \printfield{number}%
    \setunit{\addcomma\space}%
    \printfield{eid}}

\DeclareFieldFormat[article,inbook,incollection]{title}{\enquote{#1}\addcomma} 

\date{Juin 2022}
\bbkannee{74\textsuperscript{e} année, 2021--2022}  % Commandes spéciales
\bbknumero{1196}                                      % Commandes spéciales

\title{Recent progress on bounds for sets with no three terms in arithmetic progression}
\subtitle{after Bloom and Sisask, Croot, Lev, and Pach, and Ellenberg and Gijswijt}

\author{Sarah Peluse}
\address{School of Mathematics \\ Institute for Advanced Study \\ 1 Einstein Drive \\
  Princeton, NJ 08540 \\ USA}
\email{speluse@princeton.edu}

\newcommand{\coloneqq}{\mathrel{\mathop:}=}
\newcommand{\eqqcolon}{=\mathrel{\mathop:}}
\newtheorem*{roughtheo}{Rough Theorem Statement}

\begin{document}

\maketitle

\section*{Introduction}

Van der Waerden's theorem \parencite{vanderWaerden1927}, one of the foundational results
of Ramsey theory, states that if the integers are partitioned into finitely many sets,
then one of these sets must contain nontrivial arithmetic progressions,
\begin{equation}
  \label{eq:ap}
  x,x+y,\dots,x+(k-1)y,
\end{equation}
of all lengths. Here \emph{nontrivial} means that $y\neq 0$ in~\eqref{eq:ap}. Motivated
by van der Waerden's result, \textcite{ErdosTuran1936} conjectured that every subset of
the integers with positive upper density must contain arithmetic progressions of all
lengths, or, equivalently, that any subset $A$ of the first $N$ integers containing no
$k$-term arithmetic progressions satisfies $|A|=o_k(N)$. Thus, van der Waerden's theorem
should hold because, in any finite partition of the integers, some part must have positive
density.

Since any two distinct integers form a two-term arithmetic progression, the first
nontrivial case of Erd\H{o}s and Tur\'an's conjecture is when $k=3$. Define $r_3(N)$
to be the size of the largest subset of the first $N$ integers containing no nontrivial
arithmetic progressions, so that the $k=3$ case of the conjecture is equivalent to
$r_3(N)=o(N)$. This was first proven by \textcite{Roth1953}, who even produced an explicit
bound for $r_3(N)$, using a variant of the circle method.
\begin{theo}[\cite{Roth1953}]
  \label{thm:Roth}
  We have
  \begin{equation*}
    r_3(N)=O\left(\frac{N}{\log\log{N}}\right).
  \end{equation*}
\end{theo}
\textcite{Szemeredi1975} proved Erd\H{o}s and Tur\'an's conjecture in full generality via
a purely combinatorial argument in which he introduced his famous regularity lemma for
graphs, now a fundamental tool in graph theory. There are now many proofs of Szemer\'edi's
Theorem, most notably Furstenberg's proof using ergodic
theory \parencite{Furstenberg1977}, in which he introduced his famous correspondence
principle and launched the field of ergodic Ramsey theory, and Gowers's proof of explicit
quantitative bounds in Szemer\'edi's theorem \parencite{Gowers1998,Gowers2001}, which
initiated the study of higher-order Fourier analysis.

We will, for the remainder of this exposition, mostly restrict our discussion to sets
lacking three-term arithmetic progressions. It is now a central open problem in additive
combinatorics to determine the best possible bounds in Roth's theorem, i.e., to
determine the size of the largest subset of the first $N$ integers containing no nontrivial
three-term arithmetic progressions. This problem has catalyzed many important developments
in additive and extremal combinatorics, spurring the invention of techniques that have had
wide-ranging applications.

Beginning around the 1940's, Erd\H{o}s repeatedly posed the conjecture that any
subset~$S$ of the natural numbers satisfying
\begin{equation*}
%  \label{eq:recip}
  \sum_{n\in S}\frac{1}{n}=\infty
\end{equation*}
must contain arithmetic progressions of all lengths. It was also a very old, folklore
conjecture that the primes contain arbitrarily long arithmetic progressions, and Erd\H{o}s
was interested in whether the primes (whose sum of reciprocals diverges) must contain
arbitrarily long arithmetic progressions simply because they are sufficiently dense.  This
folklore conjecture is now known to be true thanks to celebrated work of
\textcite{GreenTao2008}, who leveraged the pseudorandomness of the primes in their
proof. Upper density and the divergence rate of $\sum_{n\in S}\frac{1}{n}$ are not quite
equivalent notions of size, but, by partial summation, a bound of the quality
$O_k\left(\frac{N}{(\log{N})^{1+c}}\right)$, where $c>0$, for the size of the largest
subset of the first $N$ integers containing no $k$-term arithmetic progressions would be
sufficient to prove Erd\H{o}s's conjecture. Over the past few decades, a sequence of works
had improved Roth's bound right up to the $O\left(\frac{N}{\log{N}}\right)$ barrier. The
table below summarizes these developments, where the second column lists bounds for the
order of magnitude of $r_3(N)$ obtained by the authors in the first column.

\begin{center}
\begin{tabular}{ c|c } 
% \hline
  \textcite{Roth1953} & $\frac{N}{\log\log{N}}$  \\
 \textcite{HeathBrown1987} and \textcite{Szemeredi1990} & $\frac{N}{(\log{N})^c}$  \\ 
  \textcite{Bourgain1999} & $\frac{N}{(\log{N})^{1/2-o(1)}}$  \\
  \textcite{Bourgain2008} & $\frac{N}{(\log{N})^{2/3-o(1)}}$ \\
  \textcite{Sanders2012} & $\frac{N}{(\log{N})^{3/4-o(1)}}$ \\
  \textcite{Sanders2011} & $\frac{N(\log\log{N})^6}{\log{N}}$ \\
  \textcite{Bloom2016} & $\frac{N(\log\log{N})^4}{\log{N}}$ \\
  \textcite{Schoen2021} & $\frac{N(\log\log{N})^{3+o(1)}}{\log{N}}$ \\
% \hline
\end{tabular}
\end{center}
Here the $c$ appearing in the second row is a small positive constant, the $-o(1)$ in the
exponent of $\log{N}$ in the third, fourth, and fifth rows hides bounded powers of
$\log\log{N}$ in the numerator, and the $o(1)$ in the exponent of $\log\log{N}$ in the
last row hides a bounded power of $\log\log\log{N}$.

Schoen's record upper bound for $r_3(N)$ appeared on the arXiv in May of 2020. Two
months later, \textcite{BloomSisask2020} announced that they had finally broken through
the $O\left(\frac{N}{\log{N}}\right)$ barrier in Roth's theorem, thus proving the first
nontrivial case of Erd\H{o}s's conjecture.
\begin{theo}[\cite{BloomSisask2020}]
  \label{thm:bs}
  There exists an absolute constant $c>0$ such that
  \begin{equation*}
    r_3(N)=O\left(\frac{N}{(\log{N})^{1+c}}\right).
  \end{equation*}
\end{theo}
Therefore, any set $S$ of natural numbers satisfying $\sum_{n\in S}\frac{1}{n}=\infty$
must contain a three-term arithmetic progression. Such sets include positive density
subsets of the prime numbers, so that Theorem~\ref{thm:bs} also implies Green's Roth
theorem in the primes~\parencite{Green2005R}.

We will now briefly discuss the known lower bounds for $r_3(N)$. By considering the
integers whose ternary expansion contains no twos, it is easy to see that
$r_3(N)=\Omega(N^{\log{2}/\log{3}})$. \textcite{SalemSpencer1942} constructed subsets of the first $N$ integers
of density $\exp(-\log{N}/\log\log{N})$ lacking three-term arithmetic progressions, showing
that the true order of magnitude of $r_3(N)$ is larger than $N^{1-\varepsilon}$ for any fixed
\mbox{$\varepsilon>0$.} For this reason, sets free of three-term arithmetic progressions are
sometimes called Salem--Spencer sets. A construction of \textcite{Behrend1946} shows that
$r_3(N)=\Omega(N/\exp(C\sqrt{\log{N}}))$ for some absolute constant $C>0$, which is still
essentially the best known lower bound.

There is, then, the natural question of whether the true order of magnitude of $r_3(N)$ is
closer to Behrend's lower bound or the upper bound of Bloom and Sisask. \textcite{SchoenSisask2016}
have proven bounds of the form $O\left(N/\exp(C(\log{N})^{1/7})\right)$ for subsets of the
first $N$ integers having no nontrivial solutions to the equation $x+y+z=3w$. Since
three-term arithmetic progressions are solutions to the equation $x+y=2z$, it is
reasonable to guess, by analogy, that $r_3(N)$ should also be on the order of
$N/\exp(C(\log{N})^c)$ for some absolute constants $C,c>0$. Experts have, for a while,
thought that a bound of this form is closer to the truth than, say,
$\frac{N}{(\log{N})^{100}}$, though it appears no one was brave enough to write down a conjecture. Bloom
and Sisask have finally conjectured this in their paper, and they do not just reason by
analogy--several of the steps of their proof are efficient enough to produce a bound of the
form $O\left(N/\exp(C(\log{N})^c)\right)$.

When $G$ is a finite abelian group of odd order, it is also natural to define $r_3(G)$ to
be the size of the largest subset of $G$ containing no nontrivial three-term arithmetic
progressions, and to ask for upper and lower bounds on $r_3(G)$. Obtaining bounds for
$r_3(\mathbf{Z}/M\mathbf{Z})$ as $M$ tends to infinity is essentially equivalent to
obtaining bounds in Roth's theorem in the integer setting. Another family of groups of
great interest are the finite dimensional $\mathbf{F}_3$-vector spaces. Subsets of
$\mathbf{F}_3^n$ lacking three-term arithmetic progressions are called \emph{cap-sets},
and the problem of bounding $r_3(\mathbf{F}_3^n)$, known as the \emph{cap-set problem},
has an old history. Nontrivial three-term arithmetic progressions are exactly the lines in~$\mathbf{F}_3^n$, and, more generally, sets (in finite, real, or complex affine or
projective space) with no-three-on-a-line are popular objects of study in discrete and
combinatorial geometry.

\textcite{BrownBuhler1982} were the first to prove $r_3(\mathbf{F}_3^n)=o(3^n)$. This
fact, like $r_3(N)=o(N)$, is also a straightforward consequence of the \emph{triangle
  removal lemma}, which states that, for every $\varepsilon>0$, there exists a $\delta>0$
such that any graph on $M$ vertices containing $\delta M^3$ triangles can be made
triangle-free by removing at most $\varepsilon M^2$ edges. This was observed by
\textcite{FranklGrahamRodl1987}, who then asked whether there exists a positive constant
$c<3$ such that $r_3(\mathbf{F}_3^n)=O(c^n)$. \textcite{AlonDubiner1993} also posed this
question. By adapting Roth's argument to the setting of $\mathbf{F}_3$-vector spaces,
\textcite{Meshulam1995} proved the first explicit bounds for the size of cap-sets.

\begin{theo}[\cite{Meshulam1995}]
  \label{thm:Meshulam}
  We have
  \begin{equation*}
    r_3(\mathbf{F}_3^n)=O\left(\frac{3^n}{n}\right).
  \end{equation*}
\end{theo}
The quantity $3^n$, which is the size of $\mathbf{F}_3^n$, is analogous to the length $N$
of the interval $\{1,\dots,N\}$ in Roth's theorem. Thus, Meshulam's result corresponds to a
bound of the strength $O\left(\frac{N}{\log{N}}\right)$ in Roth's theorem.

The family of vector spaces $(\mathbf{F}_3^n)_{n=1}^\infty$ can serve as a useful testing
ground for ideas and techniques to improve Roth's theorem in the integer setting, since
many technical aspects are greatly simplified when working in $\mathbf{F}_3^n$. The
surveys by \textcite{Green2005} and \textcite{Wolf2015} give nice overviews of this
philosophy. The setting of vector spaces over finite fields is often referred to in
additive combinatorics as the ``finite field model setting'', and we will also use this
terminology. In breakthrough work, \textcite{BatemanKatz2012} proved that
$r_3(\mathbf{F}_3^n)=O\left(\frac{3^n}{n^{1+c}}\right)$ for some absolute constant $c>0$,
and their insights obtained in the finite field model setting were crucial in the work
of \textcite{BloomSisask2020} in the integer setting.

Up until a few years ago, all quantitative improvements to the arguments of Roth and
Meshulam were (increasingly more difficult and technical) refinements of Roth's original
Fourier-analytic argument. In 2016, \textcite{CrootLevPach2016} introduced a new version
of the polynomial method, which they used to prove that any subset of
$(\mathbf{Z}/4\mathbf{Z})^n$ lacking three-term arithmetic progressions has cardinality at
most $O(3.61^n)$, greatly improving upon the previous best bound of
$O\left(\frac{4^n}{n(\log{n})^{c}}\right)$ due to \textcite{Sanders2009}.  Very shortly
after, \textcite{EllenbergGijswijt2017} adapted the method of Croot, Lev, and Pach to
prove a power-saving bound for the size of cap-sets, thus answering the question of Frankl, Graham,
and R\"odl.
\begin{theo}[\cite{EllenbergGijswijt2017}]
  \label{thm:eg}
  We have
  \begin{equation*}
    r_3(\mathbf{F}_3^n)=O(2.756^n).
  \end{equation*}
\end{theo}
The arguments of Croot--Lev--Pach and Ellenberg--Gijswijt are completely disjoint from the
prior Fourier-analytic arguments, and constitute yet another instance of the polynomial
method producing an elegant solution to a famous problem, joining (among other works)
Dvir's solution of the finite field Kakeya problem \parencite{Dvir2009} and the work of
Guth and Katz on the joints problem \parencite{GuthKatz2010} and the Erd\H{o}s distinct
distances problem \parencite{GuthKatz2015}. \textcite{Edel2004} has constructed cap-sets in
$\mathbf{F}_3^n$ of size $\Omega(2.217^n)$, so there is still an exponential gap between
the best known upper and lower bounds for $r_3(\mathbf{F}_3^n)$.

In this exposition, we will survey the methods going into the two breakthrough results
stated in Theorems~\ref{thm:bs} and~\ref{thm:eg}. We will begin by introducing Roth's
basic method in the finite field model and integer settings in Section~\ref{sec:Roth}, and
then give an overview of most of the ingredients in Bloom and Sisask's argument in
Section~\ref{sec:Background} before discussing their proof, with a focus on spectral
boosting, in Section~\ref{sec:BS}. We will then present a full proof of
Theorem~\ref{thm:eg} in Section~\ref{sec:EG}.

\section*{Acknowledgements}
I would like to thank Thomas Bloom, Jordan Ellenberg, Ben Green, Olof Sisask, Kannan
Soundararajan, Avi Wigderson, and Nicolas Bourbaki for helpful comments on earlier
drafts. I am supported by the NSF Mathematical Sciences Postdoctoral Research Fellowship Program under Grant No. DMS-1903038 and the Oswald Veblen fund

\section{The density-increment method and Roth's theorem}\label{sec:Roth}

We begin by fixing notation and normalizations. Along with the standard asymptotic
notation $O,\Omega,$ and $o$, we will frequently use Vinogradov's notation $\ll,\gg,$ and
$\asymp$. As a reminder, for any nonnegative real numbers $A$, $B$, $A'$, and $B'$, the
relations $A=O(B)$, $B=\Omega(A)$, $A\ll B$, and $B\gg A$ all mean that $A\leq CB$ for
some absolute constant $C>0$, and $A'\asymp B'$ means that both $A'\ll B'$ and $B'\ll
A'$. We will write $O(B)$ to denote a quantity that is $\ll B$ and $\Omega(A)$ to denote a
quantity that is $\gg A$. For any $\alpha>0$, we will write $A\lesssim_\alpha B$ to mean
that $A=O\left(\log(1/\alpha)^CB\right)$ for some absolute constant~$C$, and use
$\tilde{O}_\alpha(B)$ to denote a quantity that is $\lesssim_{\alpha}B$. We will also use
the standard notation $[N]\coloneqq \{1,\dots,N\}$, $e(z)\coloneqq e^{2\pi i z}$, and
$e_p(z)\coloneqq e(z/p)$.

Let $X$ be a finite, nonempty set, and $f\colon X\to\mathbf{C}$. The average of $f$ over $X$ is
denoted by
\[
  \mathbf{E}_{x\in X}f(x)\coloneqq \frac{1}{|X|}\sum_{x\in X}f(x).
\]
For any finite abelian group $G$, we define the $L^p$ and $\ell^p$ norms by
\begin{equation*}
  \|g\|_{L^p}^p\coloneqq \mathbf{E}_{x\in G}|g(x)|^p\qquad\text{and}\qquad\|g\|_{\ell^p}^p\coloneqq \sum_{x\in G}|g(x)|^p,
\end{equation*}
respectively, whenever $g\colon G\to\mathbf{C}$. Let $\widehat{G}$ denote the set of
characters of $G$. For any $h\colon G\to\mathbf{C}$ and $\xi\in \widehat{G}$, we define
the Fourier coefficient of~$h$ at~$\xi$ by
\[
  \widehat{h}(\xi)\coloneqq \mathbf{E}_{x\in G}h(x)\overline{\xi(x)}
\]
and the inverse Fourier transform for $F\colon\widehat{G}\to\mathbf{C}$ by
\begin{equation*}
  \widecheck{F}(x)\coloneqq \sum_{\xi\in \widehat{G}}F(\xi)\xi(x).
\end{equation*}
With this choice of normalization, the Fourier inversion formula and Plancherel's theorem
are
\[
  h(x)=\sum_{\xi\in \widehat{G}}\widehat{h}(\xi)\xi(x)\qquad\text{and}\qquad\mathbf{E}_{x}g(x)\overline{h(x)}=\sum_{\xi\in\widehat{G}}\widehat{g}(\xi)\overline{\widehat{h}(\xi)},
\]
respectively. We normalize the inner product by $\langle g,h\rangle\coloneqq\mathbf{E}_{x\in
  G}g(x)\overline{h(x)}$, convolution by $(g*h)(x)\coloneqq \mathbf{E}_{y\in G}g(x-y)h(y)$,
so that $\widehat{g*h}=\widehat{g}\cdot\widehat{h}$, and,
following \textcite{BloomSisask2020}, also define $g\circ h\coloneqq g*h_{-}$, where
$h_{-}(x)\coloneqq \overline{h(-x)}$.

For $G$ a finite abelian group and $A\subset G$, we denote the density of $A$ in $G$ by
$\mu_G(A)\coloneqq |A|/|G|$, and sometimes drop the subscript when the ambient group is clear. When $A$ is
nonempty, we will also denote the normalized indicator function of $A$ by
$\mu_A\coloneqq \frac{1}{\mu(A)}1_A$.

\subsection{The density-increment method for three-term arithmetic progressions}

Every improvement over Roth's bound for $r_3(N)$ has been based on Roth's original
argument. In this section, we will review his method (in a more modern formulation),
giving full proofs of Theorems~\ref{thm:Roth} and~\ref{thm:Meshulam}.

Roth's proof proceeds by a downward induction on density which, slightly rephrased, has
become a standard technique in additive combinatorics known as the
\emph{density-increment method}. The basic idea of the argument is that a subset of
$[N]$ or $\mathbf{F}_3^n$ either has many three-term arithmetic progressions, or else the
set has particularly large density on some nice, ``structured'' subset of $[N]$ or
$\mathbf{F}_3^n$. The structured subset resembles $[N]$ or $\mathbf{F}_3^n$ closely enough
that one can repeat the argument, except now with a subset of greater density. Since
density cannot go above one, such an iteration must terminate, at which point the set
under consideration must contain many three-term arithmetic progressions. We can then
retrace the steps of the iteration to derive an upper bound for the density of any set
lacking three-term arithmetic progressions. When working in~$\mathbf{F}_3^n$, the
structured subsets are subspaces of bounded codimension, and when working in $[N]$ or
$\mathbf{Z}/N\mathbf{Z}$, the structured subsets are either long arithmetic progressions
or (regular) Bohr sets of bounded rank.

In both the proof of Roth's theorem and the proof of Meshulam's theorem, we will derive a
density-increment when a set lacks three-term arithmetic progressions by using the
following Fourier-analytic identity: If $G$ is an abelian group and
$f,g,h\colon G\to\mathbf{C}$, then
\begin{equation}
  \label{eq:fourier}
  \mathbf{E}_{x,y\in G}f(x)g(x+y)h(x+2y)=\sum_{\xi\in\widehat{G}}\widehat{f}(\xi)\widehat{g}(-2\xi)\widehat{h}(\xi).
\end{equation}
This can easily be shown by inserting the Fourier inversion formula for the functions $f$,
$g$, and $h$ on the left-hand side and using orthogonality of characters.

\subsection{Meshulam's theorem}

We will present the proof of Meshulam's theorem before that of Roth's theorem, since the
technical details are simpler in the finite field model setting. The argument relies on
the following density-increment lemma.
\begin{theo}
  \label{thm:Meshulaminc}
Set $N\coloneqq 3^n$, and let $A\subset \mathbf{F}_3^n$ be a cap-set of density $\alpha$. Then either
\begin{equation}
  \label{eq:Nlarge}
  N<\frac{2}{\alpha^2},
\end{equation}
or there exists an affine subspace $H$ of $\mathbf{F}_3^n$ of codimension $1$ on which $A$
has density substantially larger than $\alpha$:
\begin{equation}
  \label{eq:densityinc}
  \frac{|A\cap H|}{|H|}\geq\alpha+\frac{\alpha^2}{4}.
\end{equation}
\end{theo}
\begin{proof}
  Suppose that~\eqref{eq:Nlarge} fails to hold, so that $N\geq\frac{2}{\alpha^2}$. By the
  identity~\eqref{eq:fourier},
\begin{equation}
  \label{eq:Afourier}
  \mathbf{E}_{x,y\in \mathbf{F}_3^n}1_A(x)1_A(x+y)1_A(x+2y)=\alpha^3+\sum_{0\neq \xi\in\mathbf{F}_3^n}\widehat{1_A}(\xi)^2\widehat{1_A}(-2\xi),
\end{equation}
while, since $A$ is a cap-set,
\begin{equation*}
  \mathbf{E}_{x,y\in \mathbf{F}_3^n}1_A(x)1_A(x+y)1_A(x+2y)=\frac{1}{N}\mathbf{E}_{x\in\mathbf{F}_3^n}1_A(x)^3=\frac{\alpha}{N}\leq\frac{\alpha^3}{2},
\end{equation*}
which together imply that the sum over the nontrivial characters on the right-hand side
of~\eqref{eq:Afourier} must be large:
\begin{equation*}
  \left|\sum_{0\neq\xi\in\mathbf{F}_3^n}\widehat{1_A}(\xi)^2\widehat{1_A}(-2\xi)\right|\geq\frac{\alpha^3}{2}.
\end{equation*}
By the triangle inequality and Parseval's identity, there exists a nonzero
$\xi\in\mathbf{F}_3^n$ for which $\left|\widehat{1_A}(\xi)\right|\geq\alpha^2/2$. Since the
nontrivial Fourier coefficients of $1_A$ remain unchanged after adding a constant function
to $1_A$, we must have $\left|\widehat{(1_A-\alpha)}(\xi)\right|\geq\alpha^2/2$ as well. That
is,
\begin{equation*}
  \left|\mathbf{E}_{x\in \mathbf{F}_3^n}\left(1_A-\alpha\right)(x)e_3\left(\xi\cdot x\right)\right|\geq\frac{\alpha^2}{2}.
\end{equation*}

Note that the function $e_3(\xi\cdot x)$ is constant on cosets of the codimension $1$
subspace $V\coloneqq \left\{y\in \mathbf{F}_3^n \mid\xi\cdot y=0\right\}$ of
$\mathbf{F}_3^n$. Splitting the average over $x\in \mathbf{F}_3^n$ up into an average of
averages over the cosets of $V$ and applying the triangle inequality then yields
\begin{equation}
  \label{eq:cosetavg}
  \mathbf{E}_{H\in \mathbf{F}_3^n/V}\left|\mathbf{E}_{x\in H}\left(1_A-\alpha\right)(x)\right|\geq\frac{\alpha^2}{2}.
\end{equation}
On the other hand, since $A$ has density $\alpha$, the absolute-value-free version of the
sum in~\eqref{eq:cosetavg} equals zero:
\begin{equation}
  \label{eq:zero}
  \mathbf{E}_{H\in \mathbf{F}_3^n/V}\mathbf{E}_{x\in H}\left(1_A-\alpha\right)(x)=0.
\end{equation}
Adding together~\eqref{eq:cosetavg} and~\eqref{eq:zero} and using the identity
$|r|+r=2\max\left(r,0\right)$ then gives
\begin{equation*}
  \mathbf{E}_{H\in \mathbf{F}_3^n/V}\max\left(\mathbf{E}_{x\in H}\left(1_A-\alpha\right)(x),0\right)\geq\frac{\alpha^2}{4}.
\end{equation*}
By the pigeonhole principle, there must exist some coset $H$ of $V$ such that
\begin{equation*}
  \mathbf{E}_{x\in H}\left(1_A-\alpha\right)(x)\geq\frac{\alpha^2}{4}.
\end{equation*}
Since $\mathbf{E}_{x\in H}\left(1_A-\alpha\right)(x)=\mathbf{E}_{x\in H}1_A(x)-\alpha$,
adding $\alpha$ to both sides of the above yields~\eqref{eq:densityinc}.
\end{proof}

Observe that three-term arithmetic progressions are invariant under affine-linear
transformations, in that if $S\colon V_1\to V_2$ is an affine-linear transformation and
$x,x+y,x+2y$ is a three-term arithmetic progression in~$V_1$, then $S(x),S(x+y),S(x+2y)$
is a three-term arithmetic progression in $V_2$. Further, if $S=v_2+T$ for some invertible
linear transformation~$T$ and vector $v_2\in V_2$, then $S$~maps nontrivial three-term
arithmetic progressions to non-trivial three-term arithmetic progressions. It therefore
follows that if $H$~is a coset of~$V$ in~$\mathbf{F}_3$ of dimension~$m$ and $B\subset H$
is a subset of density~$\beta$ in~$H$ containing no nontrivial three-term arithmetic
progressions, then there exists a cap-set~$B'$ in $\mathbf{F}_3^m$ of density~$\beta$.

Now, suppose that $A\subset\mathbf{F}_3^n$ is a cap-set of density $\alpha$, and set
$A_0\coloneqq A$, $n_0\coloneqq n$, and $\alpha_0\coloneqq \alpha$. Repeatedly applying the density-increment
lemma and utilizing the above observation produces a sequence of triples
$(A_i,n_i,\alpha_i)$ satisfying
\begin{enumerate}
\item $A_i\subset\mathbf{F}_3^{n_i}$ is a cap-set of density $\alpha_i$,
\item $n_{i+1}=n_{i}-1$, and
\item $\alpha_{i+1}\geq\alpha_i+\frac{\alpha_i^2}{4}$,
\end{enumerate}
provided that $N_i\geq\frac{2}{\alpha_i^2}$. Since the density cannot exceed $1$, by the
lower bound $\alpha_{i+1}\geq\alpha_i+\frac{\alpha_i^2}{4}$, this iteration must terminate
for some $i=i_0\leq \frac{16}{\alpha}$, say. At this point, the largeness assumption on $N_i$ must
fail, so that $N_{i_0}<\frac{2}{\alpha_i^2}\leq\frac{2}{\alpha^2}$. On the other hand, since
$n_{i+1}=n_{i}-1$ for all $i<i_0$, we have $N_{i_0}=3^{n-i_0}\geq
3^{n-16/\alpha}$. Combining these upper and lower bounds, we obtain
\begin{equation*}
  3^{n}<\frac{3^{16/\alpha}}{\alpha^2/2}.
\end{equation*}
Taking $\log_3$ of both sides yields $n<16/\alpha-\log_3(\alpha^2/2)<32/\alpha$, say, so
that $\alpha\ll 1/n$, thus proving Meshulam's theorem.

\subsection{Roth's theorem}

Analogously to the finite field model setting, our proof of Roth's theorem relies on the
following density-increment lemma.
\begin{theo}
  \label{thm:Rothinc}
  Let $A$ be a subset of $[N]$ of density $\alpha$ containing no nontrivial three-term
  arithmetic progressions. Then either
  \begin{equation}
    \label{eq:RothNlarge}
    N<\frac{8}{\alpha^2},
  \end{equation}
  or there exists a long arithmetic progression $P=a+q[N']$, with
  $N'\geq \alpha^4\sqrt{N}/2^{21}$, on which $A$ has density substantially larger than
  $\alpha$:
  \begin{equation*}
    \frac{|A\cap P|}{|P|}\geq\alpha+\frac{\alpha^2}{2^{11}}.
  \end{equation*}
\end{theo}
Before proving this result, we will recall Dirichlet's theorem on Diophantine
approximation, which is a simple consequence of the pigeonhole principle.

\begin{theo}
  \label{thm:Dirichlet}
  Let $\gamma_1,\dots,\gamma_k$ be real numbers. For any positive integer $Q$, there exist
  integers $p_1,\dots,p_k$ and $1\leq q\leq Q$ such that
  \begin{equation*}
    \left|\gamma_i-\frac{p_i}{q}\right|<\frac{1}{qQ^{1/k}}
  \end{equation*}
  for all $1\leq i\leq k$.
\end{theo}

Now we can prove Theorem~\ref{thm:Rothinc}.
\begin{proof}
  Suppose that~\eqref{eq:RothNlarge} fails to hold, so that $N\geq\frac{8}{\alpha^2}$. We
  begin by letting $p$ be any prime number between $2N$ and $4N$, which must exist by
  Bertrand's postulate, and noting that any three-term arithmetic progression in $[N]$
  viewed as a subset of $\mathbf{Z}/p\mathbf{Z}$ corresponds to a genuine three-term
  arithmetic progression in $[N]$. Thus, the number of three-term arithmetic progressions
  in $A$ equals
  \begin{equation}
    \label{eq:3apcount}
    \sum_{x,y\in\mathbf{Z}/p\mathbf{Z}}1_A(x)1_A(x+y)1_A(x+2y).
  \end{equation}
  Letting $f_A\coloneqq 1_A-\alpha 1_{[N]}$ denote the balanced function of
  $A$,~\eqref{eq:3apcount} can be
  written as the sum of the three terms,
  \begin{equation}
    \label{eq:term1}
    \sum_{x,y\in\mathbf{Z}/p\mathbf{Z}}1_A(x)1_A(x+y)f_A(x+2y),
  \end{equation}
  \begin{equation}
    \label{eq:term2}
    \alpha\sum_{x,y\in\mathbf{Z}/p\mathbf{Z}}1_A(x)f_A(x+y)1_{[N]}(x+2y),
  \end{equation}
  and
  \begin{equation}
    \label{eq:term3}
    \alpha^2\sum_{x,y\in\mathbf{Z}/p\mathbf{Z}}1_A(x)1_{[N]}(x+2y).
  \end{equation}
  The quantity~\eqref{eq:term3} is at least $\alpha^3N^2/4\geq 2\alpha N$. On the other
  hand, by assumption,~\eqref{eq:3apcount} equals $|A|=\alpha N$, so that at least one of the
  terms~\eqref{eq:term1} or~\eqref{eq:term2} must have magnitude at least
  $\alpha^3N^2/8\geq \alpha^3p^2/128$. Arguing as in the finite field model setting, it follows that there
  exists a nonzero integer $1\leq \xi\leq p-1$ such that
  \begin{equation}
    \label{eq:largefourierint}
    \left|\sum_{x\in\mathbf{Z}/p\mathbf{Z}}f_A(x)e\left(\frac{\xi x}{p}\right)\right|\geq\frac{\alpha^2}{2^{7}}p.
  \end{equation}

  Now we apply Dirichlet's theorem with $Q=\left\lceil\sqrt{p}\right\rceil$ to get that there exist integers $a$
  and $1\leq q\leq Q$ and a real number $0\leq \theta<1$ for which
  \begin{equation*}
    \frac{\xi}{p}=\frac{a}{q}+\frac{\theta}{q\sqrt{p}}.
  \end{equation*}
  The group $\mathbf{Z}/p\mathbf{Z}$ can be partitioned into at least
  $2^{10}\lfloor\sqrt{p}\rfloor/\alpha^2$ arithmetic progressions $P_1,\dots,P_K$ modulo
  $p$ of length $N'\coloneqq \lceil \alpha^2\sqrt{p}/2^{10}\rceil$ and common
  difference~$q$, along with~$q$ (possibly empty) arithmetic progressions
  $P_1',\dots,P_q'$ modulo $p$ of length at most $N'-1$ and common difference~$q$. It therefore
  follows from~\eqref{eq:largefourierint} that
  \begin{equation*}
    \sum_{i=1}^K\left|\sum_{x\in P_i}f_A(x)e\left(\frac{\theta x}{q\sqrt{p}}\right)\right|+\sum_{j=1}^q\left|\sum_{x\in P_j'}f_A(x)e\left(\frac{\theta x}{q\sqrt{p}}\right)\right|\geq\frac{\alpha^2}{2^7}p.
  \end{equation*}
  Note that $e(\theta x/q\sqrt{p})$ and $e(\theta y/q\sqrt{p})$ differ by a quantity of
  magnitude at most $\alpha^2/2^8$ for all pairs $x,y\in P_i$ or $x,y\in P_j'$. Thus,
  \begin{equation*}
    \sum_{i=1}^K\left|\sum_{x\in P_i}f_A(x)\right|+\sum_{j=1}^q\left|\sum_{x\in P_j'}f_A(x)\right|\geq\frac{\alpha^2}{2^8}p.
  \end{equation*}
  As in the finite field model setting, since $P_1,\dots,P_K,P_1',\dots,P_q'$ partition
  $\mathbf{Z}/p\mathbf{Z}$, combining the above with the fact that $f_A$ has mean zero on
  $\mathbf{Z}/p\mathbf{Z}$ yields
  \begin{equation*}
    \sum_{i=1}^K\max\left(\sum_{x\in P_i}f_A(x),0\right)+\sum_{j=1}^q\max\left(\sum_{x\in P_j'}f_A(x),0\right)\geq\frac{\alpha^2}{2^9}p.
  \end{equation*}
  The contribution of the second sum on the left-hand side of the above is at most
  $\alpha^2q\sqrt{p}/2^{10}<\alpha^2p/2^{10}$, so that
  \begin{equation*}
    \sum_{i=1}^K\max\left(\sum_{x\in P_i}f_A(x),0\right)\geq\frac{\alpha^2}{2^{10}}p.
  \end{equation*}

  By the pigeonhole principle, there is an $1\leq i\leq K$ such that
  $\frac{|A\cap P_i|}{|P_i|}\geq \alpha+\frac{\alpha^2}{2^{10}}$. The progression $P_i$ is
  an arithmetic progression in $\mathbf{Z}/p\mathbf{Z}$, not in the integers, so it
  remains to find a density-increment on an integer arithmetic progression. Note that
  $qN'<p$, so $P_i$ is the union $P_i=R\cup S$ of two disjoint arithmetic progressions in
  $[p]$ with common difference $q$. We may, without loss of generality, assume that
  $|R|\geq |S|$. The set $A$ must certainly have density at least
  $\alpha+\frac{\alpha^2}{2^{11}}$ on at least one of $R$ or $S$. If
  $|S|\geq \frac{\alpha^2}{2^{11}}N'$, then both $R$ and $S$ are sufficiently large and we
  have the desired density-increment on at least one of them. If
  $|S|<\frac{\alpha^2}{2^{11}}N'$, then $|R|\geq N'/2$, say, and
  $|A\cap R|\geq\left(\alpha+\frac{\alpha^2}{2^{11}}\right)N'$, so that
  $\frac{|A\cap R|}{|R|}\geq\alpha+\frac{\alpha^2}{2^{11}}$ since $|R|\leq N'$ and we
  again have the desired density-increment.
\end{proof}

Analogously to the finite field model setting, observe that three-term arithmetic
progressions are translation-dilation invariant, so that if $B$ contains no nontrivial
three-term arithmetic progressions, then $B'\coloneqq \{n\in[N']\mid a+qn\in B\cap P\}$ has density
$\frac{|B\cap P|}{|P|}$ in $[N']$ and also contains no nontrivial three-term arithmetic
progressions.

Now, suppose that $A\subset[N]$ has density $\alpha$ and contains no nontrivial three-term
arithmetic progressions, and set $A_0\coloneqq A$, $N_0\coloneqq  N$, and $\alpha_0\coloneqq \alpha$. Repeated
applications of the density-increment lemma produces a sequence of triples
$(A_i,N_i,\alpha_i)$ satisfying
\begin{enumerate}
\item $A_i\subset[N_i]$ has density $\alpha_i$ and contains no nontrivial three-term
  arithmetic progressions,
\item $N_{i+1}\geq \alpha^4\sqrt{N_i}/2^{21}$, and
\item $\alpha_{i+1}\geq\alpha_i+\frac{\alpha_i^2}{2^{11}}$,
\end{enumerate}
provided that $N_i\geq\frac{8}{\alpha_i^2}$. As in the density-increment for Meshulam's
theorem, this iteration must terminate for some $i_0\ll\frac{1}{\alpha}$, at which point
the largeness assumption must fail, so that $N_{i_0}<\frac{8}{\alpha^2}$. On the other
hand, we have $N_{i_0}\gg\alpha^{8}N^{1/2^{i_0}}\gg\alpha^8
N^{1/2^{O(1/\alpha)}}$. Combining these upper and lower bounds yields
\begin{equation*}
  N^{1/2^{O(1/\alpha)}}\ll\frac{1}{\alpha^{10}},
\end{equation*}
from which Roth's theorem follows by taking the double logarithm of both sides when $N$~is
sufficiently large.

\section{Key ingredients from prior quantitative improvements}\label{sec:Background}
Inspecting the proofs of Roth's theorem and Meshulam's theorem, we see that we obtained
worse bounds in the former because the structured set on which we found a
density-increment shrinks much more rapidly ($N_{i+1}\asymp\alpha^{O(1)} \sqrt{N_i}$) in
the integer setting than in the finite field model setting ($N_{i+1}\asymp N_i$). Thus,
Theorem~\ref{thm:Rothinc} is much less efficient than Theorem~\ref{thm:Meshulaminc} to
iterate. Therefore, for a long time, the goal of much of the work on quantitative bounds
in Roth's theorem had been to obtain density-increment results in the integer setting that
are as efficient as that obtained in Theorem~\ref{thm:Meshulaminc}. This eventually led to
four different proofs of the bound $r_3(N)\ll\frac{N}{(\log{N})^{1-o(1)}}$. The argument
of Bloom and Sisask relies on many insights made in these prior works, along
with those that allowed Bateman and Katz to go beyond the
$O\left(\frac{N}{\log{N}}\right)$ bound in the cap-set problem. The goal of this section
is to summarize these insights and introduce the related concepts needed to understand
Bloom and Sisask's proof.

\subsection{Obtaining a density-increment from large $\ell^2$-energy}
The key insight of \textcite{HeathBrown1987} and \textcite{Szemeredi1990} was that if
$f_A$ has several large Fourier coefficients, then it is more efficient to do one large
density-increment step using all of these coefficients than to do individual
density-increment steps for each of them. To be more precise, the
starting point of their argument is to show that if $A$ contains no nontrivial three-term
arithmetic progressions, then a large proportion of the $\ell^2$-mass of $\widehat{f_A}$
can be captured in a relatively small number of nontrivial Fourier coefficients. Recall
from the proof of Theorem~\ref{thm:Rothinc} that either
\begin{equation*}
  \left|\sum_{0\neq \xi\in\mathbf{Z}/p\mathbf{Z}}\widehat{1_{[N]}}(\xi)^2\widehat{f_A}(-2\xi)\right|>\frac{\alpha}{8}
\end{equation*}
or
\begin{equation*}
  \left|\sum_{0\neq \xi\in\mathbf{Z}/p\mathbf{Z}}\widehat{1_{[N]}}(\xi)\widehat{f_A}(\xi)\widehat{1_{[N]}}(-2\xi)\right|>\frac{\alpha}{8},
\end{equation*}
provided that $N$ is sufficiently large in terms of $\alpha$. In either case, it follows
from H\"older's inequality that
$\|\widehat{1_{[N]}}\|_{\ell^3}^2\|\widehat{f_{A}}\|_{\ell^3}\gg\alpha$, so that
$\|\widehat{f_A}\|_{\ell^3}^3\gg\alpha^3$ since $\|\widehat{1_{[N]}}\|_{\ell^3}\ll
1$. Thus, using the layercake representation and the fact that
$|\widehat{f_A}(\xi)|\leq2\alpha$ for all $\xi\in\mathbf{Z}/p\mathbf{Z}$, we have
\begin{equation*}
  \int_0^{2\alpha} z^2\cdot|\{\xi\in\mathbf{Z}/p\mathbf{Z} \mid |\widehat{f_A}(\xi)|\geq z\}|dz\gg\alpha^3.
\end{equation*}
On the other hand, if it were the case that
\begin{equation*}
  \sum_{\substack{\xi\in \mathbf{Z}/p\mathbf{Z} \\ |\widehat{f_A}(\xi)|\geq
      z}}|\widehat{f_A}(\xi)|^2\leq \frac{\alpha^2}{C}
  |\{\xi\in\mathbf{Z}/p\mathbf{Z}\mid |\widehat{f_A}(\xi)|\geq z\}|^{1/9},
\end{equation*}
say, for all $0\leq z\leq 2\alpha$, then, by bounding the left-hand side below by
$z^2\cdot|\{\xi\in\mathbf{Z}/p\mathbf{Z}\mid |\widehat{f_A}(\xi)|\geq z\}|$, we obtain
$|\{\xi\in\mathbf{Z}/p\mathbf{Z} \mid|\widehat{f_A}(\xi)|\geq z\}|\leq
\alpha^{9/4}z^{-9/4}/C^{9/8}$, which means that
\begin{equation*}
  \int_0^{2\alpha} z^2\cdot|\{\xi\in\mathbf{Z}/p\mathbf{Z} \mid |\widehat{f_A}(\xi)|\geq
  z\}|dz\leq \frac{\alpha^{9/4}}{C^{9/8}}\int_0^{2\alpha}\frac{1}{z^{1/4}}dz\ll\frac{\alpha^3}{C^{9/8}}.
\end{equation*}
Thus, choosing $C$ sufficiently large, we must have
\begin{equation*}
   \sum_{\substack{\xi\in \mathbf{Z}/p\mathbf{Z} \\ |\widehat{f_A}(\xi)|\geq
       z}}|\widehat{f_A}(\xi)|^2\gg
   \alpha^2|\{\xi\in\mathbf{Z}/p\mathbf{Z}\mid |\widehat{f_A}(\xi)|\geq z\}|^{1/9}
 \end{equation*}
 for some $0<z\leq 2\alpha$.

 Now, we enumerate the frequencies
 $\{\xi_1,\dots,\xi_m\}\coloneqq \{\xi\in\mathbf{Z}/p\mathbf{Z}\mid |\widehat{f_A}(\xi)|\geq z\}$ and
 apply Dirichlet's theorem with $Q=p^{m/(m+1)}$ to $\xi_1/p,\dots,\xi_m/p$ to produce
 integers $a_1,\dots,a_m$ and $1\leq q\leq Q$ for which
 $\left|\xi_i/p-a_i/q\right|<1/qQ^{1/m}$ for all $i=1,\dots,m$. Analogously to the proof
 of Theorem~\ref{thm:Rothinc}, we will find a density-increment on an arithmetic
 progression of common difference $q$ and length on the order of $\alpha p^{1/(m+1)}$ by an
 averaging argument. Let $P$ be any arithmetic progression of common difference $q$ and
 length $p^{1/(m+1)}/10$, say, and consider the second moment
 $\mathbf{E}_{x\in\mathbf{Z}/p\mathbf{Z}}(1_A*1_P)(x)^2$ of the density of
 $|A\cap (P-x)|$. We have
\begin{align*}
  \mathbf{E}_{x\in\mathbf{Z}/p\mathbf{Z}}(1_A*1_P)(x)^2 &= \sum_{\xi\in
                                                          \mathbf{Z}/p\mathbf{Z}}|\widehat{1_A}(\xi)|^2|\widehat{1_P}(\xi)|^2
  \\
  &\geq
    \alpha^2\left(\frac{|P|}{p}\right)^2+\sum_{i=1}^k|\widehat{1_A}(\xi_i)|^2|\widehat{1_P}(\xi_i)|^2
  \\
  &\geq \left(\frac{|P|}{p}\right)^2\left(\alpha^2+\Omega(\alpha^2m^{1/9})\right),
\end{align*}
where the second inequality follows from the fact that $|\widehat{1_P}(\xi_i)|\gg |P|/p$ for
all $1\leq i\leq m$. On the other hand,
\begin{equation*}
  \mathbf{E}_{x\in\mathbf{Z}/p\mathbf{Z}}(1_A*1_P)(x)^2\leq\!\!\|1_A*1_P\|_{L^\infty}\cdot\mathbf{E}_{x\in\mathbf{Z}/p\mathbf{Z}}1_A*1_P(x)=\alpha\left(\frac{|P|}{p}\right)\!\|1_A*1_P\|_{L^\infty}.
\end{equation*}
We conclude that there exists an $x\in\mathbf{Z}/p\mathbf{Z}$ for which
$\frac{p}{|P|}1_A*1_P(x)\geq \alpha(1+\Omega(m^{1/9}))$, i.e.,
$|A\cap (P-x)|/|P|\geq \alpha(1+\Omega(m^{1/9}))$. We are not quite done because $P-x$ is
a progression modulo $p$, but since $q|P|<p$, the argument given at the end of the proof
of Theorem~\ref{thm:Rothinc} guarantees that we can find a density-increment of at least
$\alpha(1+\Omega(m^{1/9}))$ on an integer arithmetic progression of length $\gg \alpha p^{1/(m+1)}$ and common
difference $q$.

The following density-increment theorem summarizes what we have shown.
\begin{theo}
  \label{thm:HBSinc}
  Let $A$ be a subset of $[N]$ of density $\alpha$ containing no nontrivial three-term
  arithmetic progressions. Then either
  \begin{equation}
    N<\frac{8}{\alpha^2},
  \end{equation}
  or else there exists an integer $1\leq m\ll\alpha^{-9}$ and a long arithmetic
  progression $P=a+q[N']$, with $N'\gg\ \alpha N^{1/(m+1)}$, on which $A$ has density
  \begin{equation*}
    \frac{|A\cap P|}{|P|}\geq \alpha(1+\Omega(m^{1/9})).
  \end{equation*}
\end{theo}
A bound of the form $r_3(N)\ll\frac{N}{(\log{N})^c}$ can now be obtained by a
straightforward adaptation of the density-increment iteration used to prove Roth's
theorem. Theorem~\ref{thm:HBSinc} is still not as efficient as
Theorem~\ref{thm:Meshulaminc}. In fact, adapting the arguments of this section to the
finite field model setting produces a worse bound for $r_3(\mathbf{F}_3^n)$ than in
Meshulam's theorem. The key idea of using large $\ell^2$-Fourier mass, instead of just one
large Fourier coefficient, to obtain a density-increment will continue to be a useful
insight, however.
  
  \subsection{Bohr sets}

  The proof of Theorem~\ref{thm:Meshulaminc} produces an efficient density-increment
  because the level sets of characters are affine subspaces of $\mathbf{F}_3^n$, which
  allows one to pass immediately from a lower bound of the form
  $|\mathbf{E}_{x\in\mathbf{F}_3^n}f_A(x)e_3(\xi\cdot x)|\gg\alpha^2$ to a
  density-increment on a large structured set. In contrast, most characters of
  $\mathbf{Z}/p\mathbf{Z}$ fluctuate too much on arithmetic progressions of length
  $\asymp p$ for us to have any hope of finding a large density-increment on such a
  progression. Thus, to remove the phase in~\eqref{eq:largefourierint}, we had to
  partition most of $\mathbf{Z}/p\mathbf{Z}$ into a many, much shorter, arithmetic
  progressions, so that $e(\xi x/p)$ was close to constant on each. The key insight
  of \textcite{Bourgain1999} was to simply partition $\mathbf{Z}/p\mathbf{Z}$ exactly into
  the sets $\{x\in\mathbf{Z}/p\mathbf{Z}\mid \|\xi x/p\|\approx z \}$ on which the character
  is close to constant, and to run the density-increment argument relative to them instead
  of relative to long arithmetic progressions.

  These approximate level sets of characters are known as \emph{Bohr sets}. Bohr sets have
  positive density in the ambient group, but behave even less like subgroups than long
  arithmetic progressions. The first useful feature of intervals and subgroups that we
  used in our earlier arguments was the ease of counting the number of three-term
  arithmetic progressions they contain. We showed in both cases that the ambient interval
  or group contained many three-term arithmetic progressions, so that, if a subset $A$
  contained few progressions, some average involving $\widehat{f_A}$ had to be large. In
  contrast, it is very difficult to count three-term arithmetic progressions in general
  Bohr sets. Thus, while he was able to obtain a density-increment on a much larger
  structured set, Bourgain had to pay the price by dealing with the poor behavior of Bohr
  sets under addition.

  We will now formally define Bohr sets and their related parameters, and then state some
  standard facts about them. Many of these can be found in \textcite{Bourgain1999}
  or \textcite[Chapter 4]{TaoVu2006}.
\begin{defi}
  Let $G$ be a finite abelian group, $\Gamma\subset\widehat{G}$ be nonempty, and
  $\nu\colon\Gamma\to[0,2]$. The \emph{Bohr set of rank $|\Gamma|$ and
  width $\nu$ with frequency set $\Gamma$} is defined as the triple
  $(\Gamma,\nu,\Bohr(\Gamma,\nu))$, where
  \begin{equation*}
    \Bohr(\Gamma,\nu)\coloneqq \{x\in G\mid |\gamma(x)-1|\leq\nu(\gamma)\text{ for all }\gamma\in\Gamma\}.
  \end{equation*}
\end{defi}
We will just refer to the Bohr set $(\Gamma,\nu,\Bohr(\Gamma,\nu))$ by
$\Bohr(\Gamma,\nu)$, even though one Bohr set can be generated by many pairs of widths and
frequency sets. Note that Bohr sets are symmetric, contain the identity, and, when
$G=\mathbf{F}_3^n$, a Bohr set of rank $r$ and constant width less than $\sqrt{3}$ is just
a subspace of codimension at most $r$.

While Bohr sets are not nearly as additively structured as long arithmetic progressions,
we still have some control over the size of their sumsets, as the following lemma shows.
\begin{lemm}
  \label{lem:bohrsum}
  We have
  \begin{equation*}
    \Bohr(\Gamma,\nu_1)+\Bohr(\Gamma,\nu_2)\subset \Bohr(\Gamma,\nu_1+\nu_2)
  \end{equation*}
  and
  \begin{equation*}
    |\Bohr(\Gamma,2\nu)|\leq 4^{|\Gamma|}|\Bohr(\Gamma,\nu)|.
  \end{equation*}
\end{lemm}
Sumsets of Bohr sets are more well-behaved when the width of one of the Bohr sets is very
small. It will therefore be useful to define, when $B=\Bohr(\Gamma,\nu)$ is a Bohr set of
width $\nu$ and $\rho>0$, the \emph{dilation} of $B$ by $\rho$ to be the Bohr set
$B_\rho\coloneqq \Bohr(\Gamma,\rho\nu)$.

Despite having some control on the size of sumsets of Bohr sets from
Lemma~\ref{lem:bohrsum}, Bohr sets can have very large doubling constant $|B+B|/|B|$ when
their rank is not extremely small. This presents a problem when attempting to run a
density-increment argument relative to a Bohr set, since we must, first of all, show that
there are many more than just the trivial three-term arithmetic progressions. If $B+B$ is
much larger than $2\cdot B$, it is not clear that we should expect there to be many
representations of elements of $2\cdot B$ as sums of two elements of $B$.

Bourgain gets around this issue by restricting the common difference of the arithmetic
progressions to lie in $B_\varepsilon$ for some small $\varepsilon$. If
$B+B_{\varepsilon}\approx B$, then it is easy to show that $B$ contains many three-term
arithmetic progressions with common difference in $B_{\varepsilon}$. If $A$ contains no
nontrivial three-term arithmetic progressions, then it certainly has none with common
difference in $B_\varepsilon$, and one can then deduce that some average involving $f_A$
over these three-term arithmetic progressions with restricted difference is large. The
Bohr sets for which we can reliably find such a dilation $B_\varepsilon$ are called
\emph{regular Bohr sets}.
\begin{defi}
  We say that a Bohr set $B$ of rank $r$ is \emph{regular} if, for all real numbers~$\delta$ satisfying $|\delta|\leq\frac{1}{100 r}$, we have
  \begin{equation*}
    (1-100r|\delta|)|B|\leq |B_{1+\delta}|\leq (1+100r|\delta|)|B|.
  \end{equation*}
\end{defi}
Not all Bohr sets are regular, but Bourgain showed that every Bohr set has many dilates
that are regular.
\begin{lemm}
  Let $B$ be a Bohr set. Then, for any $0\leq t\leq 1$, the Bohr set $B_\rho$ is regular for some $t/2\leq\rho\leq t$.
\end{lemm}

Finally, Bohr sets do indeed have constant density (depending on the width), both in the
ambient group and in Bohr supersets.
\begin{lemm}
  If $\nu'\leq\nu$, we have
  \begin{equation*}
    |\Bohr(\Gamma,\nu')|\geq\left(\prod_{\gamma\in\Gamma}\frac{\nu'(\gamma)}{4\nu(\gamma)}\right)|\Bohr(\Gamma,\nu)|.
  \end{equation*}
\end{lemm}
This implies, in particular, that $|B_\rho|\geq(\rho/4)^{\rk{B}}|B|$ for any Bohr set $B$
and dilation factor $\rho<1$.

With the introduction of Bohr sets, we have now reached the point in this exposition where
the arguments discussed are far too technical for it to be appropriate to give anything
close to full proofs. We will instead mostly highlight the key ideas, and include some
representative arguments.

So, suppose that $N$ is an odd positive integer, $B\subset\mathbf{Z}/N\mathbf{Z}$ is a regular Bohr
set, and $A\subset B$ contains no nontrivial three-term arithmetic progressions. Let
$\rho>0$ with
\begin{equation*}
  \frac{1}{800r}<\rho<\frac{1}{400r}
\end{equation*}
be such that $B_{\rho}$ is regular, and set $f_A\coloneqq 1_A-\alpha 1_B$. Then
\begin{equation*}
  \sum_{\substack{x\in \mathbf{Z}/N\mathbf{Z} \\ y\in
      B_{\rho}}}1_A(x)1_A(x+y)1_A(x+2y)=\alpha N,
\end{equation*}
while the left-hand side above can be written as
\begin{equation*}
  \alpha^3\sum_{\substack{x\in \mathbf{Z}/N\mathbf{Z} \\ y\in B_{\rho}}}1_B(x)1_B(x+y)1_B(x+2y)
\end{equation*}
plus some sums involving $f_A$. To count the number of three-term arithmetic progressions
in $B$ with common difference in $B_\rho$, note that if $x\in B_{(1-2\rho)}$, then $x+y$ and
$x+2y$ both lie in $B$ whenever $y\in B_\rho$, so that
\begin{equation*}
  \sum_{\substack{x\in \mathbf{Z}/N\mathbf{Z} \\ y\in
      B_{\rho}}}1_B(x)1_B(x+y)1_B(x+2y)\geq\sum_{z\in\mathbf{Z}/N\mathbf{Z}}1_B(z)[1_{B_{(1-2\rho)}}*1_{B_{\rho}}](z).
\end{equation*}
By the regularity of $B$, the convolution $1_{B_{(1-2\rho)}}*1_{B_\rho}$ is very close to
$|B_\rho|$ times the indicator function of $B$. Indeed, the regularity of $B$ implies that
\begin{equation*}
  \sum_{y\in\mathbf{Z}/N\mathbf{Z}}|1_{B}(y)-1_{B_{(1-2\rho)}}(y-w)|\leq 200r\rho|B|
\end{equation*}
for every $w\in B_\rho$, so that
\begin{align*}
  \sum_{z\in\mathbf{Z}/N\mathbf{Z}}\left|[1_{B_{(1-2\rho)}}*1_{B_{\rho}}](z)-|B_\rho|1_{B}(z)\right|
  &=
\!\!\!\!    \sum_{z\in\mathbf{Z}/N\mathbf{Z}}\left|\sum_{x\in\mathbf{Z}/N\mathbf{Z}}\left(1_{B_{(1-2\rho)}}(z-x)-1_{B}(z)\right)1_{B_{\rho}}(x)\right|
  \\
  &\leq
    \sum_{x\in\mathbf{Z}/N\mathbf{Z}}\sum_{z\in\mathbf{Z}/N\mathbf{Z}}\left|1_{B_{(1-2\rho)}}(z-x)-1_{B}(z)\right|1_{B_{\rho}}(x)
  \\
  &\leq
    200r\rho|B||B_\rho|\leq \frac{1}{2} N|B_{\rho}|.
\end{align*}
Thus,
\begin{equation*}
\left|\sum_{z\in\mathbf{Z}/N\mathbf{Z}}1_B(z)[1_{B_{(1-2\rho)}}*1_{B_{\rho}}](z)-N|B_\rho|\right|\leq \frac{N}{2}|B_\rho|,
\end{equation*}
from which it follows that
\begin{equation*}
  \alpha^3\sum_{\substack{x\in \mathbf{Z}/N\mathbf{Z} \\ y\in B_{\rho}}}1_B(x)1_B(x+y)1_B(x+2y)\geq\frac{\alpha^3}{2}N|B_\rho|.
\end{equation*}

As a consequence, one of the sums involving~$f_A$ must have absolute value
$\gg\alpha^3N|B_\rho|$ when $N$~is sufficiently large. The cost of being able to count the
number of three-term arithmetic progressions in~$B$ is that now the range of~$y$ in these
sums is restricted to~$B_\rho$, so we do not have the Fourier
representation~\eqref{eq:fourier} that was crucial in our previous arguments. To proceed,
the idea is to insert extra averaging in the $x$~variable with the goal of localizing~$x$
to a (translate of) an even smaller regular dilate of~$B$, and approximate the sums using
the regularity of the various Bohr sets floating around until the restriction that $y$~lies in a Bohr set is transformed into the restriction that $x,x+y,x+2y$ all lie in a Bohr
set, while $y$~is allowed to freely range. Then the formula~\eqref{eq:fourier} can be
applied, yielding the following density-increment result.
\begin{theo}
  \label{thm:Bourgain1inc}
  Let $N$ be an odd positive integer, $B=\Bohr(\Gamma,\nu)$ be a regular Bohr set, and $A$
  be a subset of $B$ of density $\alpha$ containing no three-term arithmetic
  progressions. Then either
  \begin{equation}
    \label{eq:Bohrupper}
    N\ll \left(\frac{\rk{B}}{\alpha}\right)^{O(\rk{B})}\prod_{\gamma\in\Gamma}\nu(\gamma)^{-1},
  \end{equation}
  or else there exists a regular Bohr set $B'\subset B$ of width $\nu'$ satisfying
  \begin{enumerate}
  \item $\rk{B'}\leq\rk{B}+1$ and
  \item $\nu'\gg\left(\frac{\alpha}{\rk{B}}\right)^{O(1)}\nu$
  \end{enumerate}
  on some translate of which $A$ has density at least $\alpha+\Omega(\alpha^2)$.
\end{theo}
Starting with a subset $A$ of $\mathbf{Z}/N\mathbf{Z}$ of density $\alpha$ containing no
nontrivial three-term arithmetic progressions and running a density-increment iteration
then produces an inequality of the form
\begin{equation}
  \label{eq:Bohrdinc}
  \alpha^{C/\alpha^2}N\leq C'
\end{equation}
for some absolute constants $C,C'>0$, from which the bound
$r_3(N)\ll\frac{N}{(\log{N})^{1/2-o(1)}}$ of \textcite{Bourgain1999} follows. The width of
the Bohr set shrinking by a factor of $(\alpha/\rk{B})^{O(1)}$ at each step of the iteration is
responsible for the exponent of $1/2$ on $\log{N}$. If the width stayed constant, as is
the case in the finite field model setting, we would have obtained a bound of the form
$r_3(N)\ll\frac{N}{(\log{N})^{1-o(1)}}$.

All quantitative improvements to Bourgain's bound have also been obtained by running a
density-increment argument relative to Bohr sets, so we will introduce (a simplification,
suitable for our expository purposes, of) a piece of notation,
from \textcite{BloomSisask2020}, that succinctly summarizes the strength of a
density-increment on a Bohr set. This notation will provide a useful way of comparing the
efficiency of different density-increment results.
\begin{defi}
  Let $B$ be a regular Bohr set of rank $r$, and $A\subset B$ have density $\alpha$ in
  $B$. We say that $A$ has a density-increment of strength $[\delta,r';C]$ relative to $B$
  if there exists a regular Bohr set $B'\subset B$ of rank
  \begin{equation*}
    \rk(B')\leq r+Cr'
  \end{equation*}
  and size
  \begin{equation*}
    |B'|\geq (2r(r'+1))^{-C(r+r')}|B|
  \end{equation*}
  for which $A$ has increased density at least
  \begin{equation*}
    \left(1+\frac{\delta}{C}\right)\alpha
  \end{equation*}
  on some translate of $B'$.
\end{defi}
For example, Theorem~\ref{thm:Bourgain1inc} says that $A$ has a density-increment of
strength $[\alpha,1;\tilde{O}_\alpha(1)]$ relative to $B$.

\subsection{The dimension of the large spectrum}

The sets of frequencies at which $\widehat{1_A}$ is large, which we considered in the proof
of Theorem~\ref{thm:HBSinc}, are called the \emph{large spectra} of $A$.
\begin{defi}
  Let $G$ be an abelian group, $A\subset G$ be a subset of density $\alpha$, and
  $\delta>0$. The $\delta$-large spectrum of $A$ is the set
  \begin{equation*}
    \Spec_\delta(A)\coloneqq \{\xi\in\widehat{G}\mid |\widehat{1_A}(\xi)|\geq\delta\alpha\}.
  \end{equation*}
\end{defi}
Note that $\lvert\Spec_\delta(A)\rvert\leq 1/(\alpha\delta^2)$ for all $\delta>0$ by Parseval's
identity.

Suppose that $A\subset\mathbf{Z}/N\mathbf{Z}$ has density $\alpha$ and contains no nontrivial three-term arithmetic
progressions, and set $f_A\coloneqq 1_A-\alpha$, so that, as in the proof of Theorem~\ref{thm:HBSinc},
$\|\widehat{f_A}\|_{\ell^3}^3\gg\alpha^3$. By dyadic pigeonholing, there exists some $1\geq\delta\gg\alpha$ such that
\begin{equation}
  \label{eq:largel3}
  \sum_{\xi\in\Spec_{\delta}(A)\setminus\Spec_{2\delta}(A)}|\widehat{f_A}(\xi)|^3\gtrsim_\alpha \alpha^3.
\end{equation}
For the benefit of the reader who has not seen dyadic pigeonholing, which is a common
argument in additive combinatorics, this is obtained by noting that Fourier coefficients
of size $\ll\alpha^2$ contribute $\ll\alpha^3$ to $\|\widehat{f_A}\|^3_{\ell^3}$ by
Parseval's identity, and then
decomposing the remaining frequencies into dyadic blocks
$\{\xi\in G\mid 2^i\alpha<|\widehat{f_A}(\xi)|\leq 2^{i+1}\alpha\}$ (of which there are
$\tilde{O}_\alpha(1)$) and applying the pigeonhole principle.

Note that if~\eqref{eq:largel3} holds, then we must have
$\lvert\Spec_{\delta}(A)\rvert\gtrsim_\alpha\delta^{-3}$, as well as that
\begin{equation*}
  2\delta\alpha\sum_{\xi\in\Spec_{\delta}(A)}|\widehat{f_A}(\xi)|^2\geq\sum_{\xi\in\Spec_{\delta}(A)\setminus\Spec_{2\delta}(A)}|\widehat{f_A}(\xi)|^3\gtrsim_\alpha \alpha^3,
\end{equation*}
so
\begin{equation}
\label{eq:largel2}
  \sum_{\xi\in\Spec_{\delta}(A)}|\widehat{f_A}(\xi)|^2\gtrsim_{\alpha}\frac{\alpha^2}{\delta}.
\end{equation}
One can now adapt the $\ell^2$-Fourier mass increment idea of Heath-Brown and Szemer\'edi
to the setting of Bohr sets to deduce a large density-increment for $A$ on a regular Bohr
set. Most papers on quantitative bounds in Roth's theorem posterior to \textcite{Bourgain1999}
contain a variant of the following standard lemma, which is essentially \textcite[Lemma~7.2]{Sanders2012}.
\begin{lemm}
  \label{lem:L2Bohr}
  Let $B$ be a regular Bohr set of rank $r$, $A\subset B$ have density $\alpha$ in $B$,
  $f_A\coloneqq 1_A-\alpha 1_B$, $K>0$ be a parameter, and $\Gamma\subset\mathbf{Z}/N\mathbf{Z}$ be a
  set of frequencies for which
  \begin{equation*}
    \sum_{\gamma\in \Gamma}|\widehat{f_A}(\gamma)|^2\geq K\alpha^2\mu(B).
  \end{equation*}
  Suppose that $B'\subset B_\rho$, where $\rho\ll\alpha K/r$, is a Bohr set of rank $r'$
  such that
  \begin{equation}
    \label{eq:gammasubset}
    \Gamma\subset\left\{\gamma\in\mathbf{Z}/N\mathbf{Z}\mid |1-\gamma(x)|\leq\frac{1}{2}\text{
      for all }x\in B'\right\}.
  \end{equation}
  Then there exists a regular Bohr set $B''$ satisfying
  \begin{enumerate}
  \item $\rk(B'')=r'$ and
  \item $\mu(B'')\geq 2^{-O(r')}\mu(B)$
  \end{enumerate}
  such that $A$ has density at least $\alpha(1+\Omega(K))$ on some translate of $B''$.
\end{lemm}
The efficiency of this density-increment result directly depends on how small we can take
$r'$ to be. So, given $\Gamma$, we want to find a Bohr set of rank as small as possible
for which~\eqref{eq:gammasubset} holds. When $\Gamma$ is an arbitrary set, the best we can
do is $\rk{B'}=|\Gamma|$. But, in our situation, $\Gamma=\Spec_\delta(A)$. Another key insight of \textcite{Bourgain2008} was that, because large spectra are
highly additively structured, one can do much better for them than the trivial bound
$\rk{B'}\leq\lvert\Spec_\delta(A)\rvert$.

There are multiple senses in which the large spectrum possesses additive structure, but the
relevant one for this section is that the large spectrum has small dimension, a result due
to \textcite{Chang2002}.
\begin{defi}
  Let $G$ be an abelian group. A subset $S\subset G$ is said to be \emph{dissociated} if
  $\sum_{s\in S}\epsilon_ss=0$ for $\epsilon_s\in\{-1,0,1\}$ only when $\epsilon_s=0$ for
  all $s\in S$. The \emph{dimension} of a set in $G$ is the size of its largest
  dissociated subset.
\end{defi}
\begin{lemm}[\cite{Chang2002}]
  \label{lem:chang}
  Let $A\subset \mathbf{Z}/N\mathbf{Z}$ be a subset of density $\alpha$, and
  $\delta>0$. Then $\dim\Spec_{\delta}(A)\lesssim_\alpha1/\delta^2$.
\end{lemm}
Lemma~\ref{lem:chang} was first used by Chang to improve the best known bounds in the
Freiman--Ruzsa theorem, and has since found many applications in additive combinatorics
and theoretical computer science. The bound
$\dim\Spec_{\delta}(A)\lesssim_\alpha1/\delta^2$ (which \textcite{Green2003} showed is sharp) should be
compared with the bound $\lvert\Spec_\delta(A)\rvert\leq 1/(\alpha\delta^2)$ from above, so that,
when $\alpha$ is small, the dimension of the large spectrum is much smaller than
its cardinality.

One can find (at the cost of shrinking $\rho$ by a factor of $(\alpha/r)^{O(1)}$, which is
not an issue) $B'$ as in Lemma~\ref{lem:L2Bohr} of rank $\ll\dim\Spec_{\delta}(A)$,
illustrating a direct connection between the additive structure of large spectra and
efficiency of density-increments. Combining this with Chang's lemma produces a
density-increment of strength $[1,1/\alpha^2;\tilde{O}_\alpha(1)]$ when
$B=\mathbf{Z}/N\mathbf{Z}$. To obtain such a density-increment when $B$ is any regular
Bohr set, one needs to work with notions of dissociativity and dimension defined relative
to Bohr sets, as well as prove a relative version of Chang's theorem. See, for example,
\textcite{Sanders2012} for more on this important technical detail. Running a
density-increment iteration then recovers the bound for $r_3(N)$ from
\textcite{Bourgain1999} up to an extra power of $\log\log{N}$.

Though the bound on dimension in Chang's lemma is sharp, \textcite{Bloom2016} proved that
one can obtain a better bound by passing to a positive density subset of the large
spectrum.
\begin{lemm}[\cite{Bloom2016}]
  \label{lem:bloom}
  Let $A\subset \mathbf{Z}/N\mathbf{Z}$ be a subset of density $\alpha$, and
  let {$\delta>0$.} Then there exists a subset $S\subset\Spec_\delta(A)$ satisfying
  $|S|\gg\delta\lvert\Spec_\delta(A)\rvert$ for which $\dim S\lesssim_\alpha1/\delta$.
\end{lemm}
Bloom actually proved a version of this lemma relativized to Bohr sets which, combined
with Lemma~\ref{lem:L2Bohr}, produces a density-increment of strength
$[1,1/\alpha;\tilde{O}_\alpha(1)]$ relative to Bohr sets. Running a density-increment iteration then yields
$r_3(N)\ll\frac{N}{(\log{N})^{1-o(1)}}$.

\subsection{Almost-periodicity of convolutions}

\textcite{Sanders2011} was the first to prove a bound of the form
$r_3(N)\ll\frac{N}{(\log{N})^{1-o(1)}}$, and he did this not by further analysis of the
additive structure of large spectra, but by utilizing methods on the ``physical
side''. \textcite{CrootSisask2010} proved a variety of theorems saying, roughly, that
convolutions are approximately translation-invariant under a large set of shifts, and
called this phenomenon~\emph{almost-periodicity}. It is possible to take the set of
shifts to be a subspace, long arithmetic progression, or Bohr set, depending on the
ambient group or the desired application. One of these almost-periodicity results was a
key input into the work of \textcite{Sanders2011}, and \textcite{BloomSisask2019} later
gave a proof of the bound $r_3(N)\ll\frac{N}{(\log{N})^{1-o(1)}}$ almost completely
relying on almost-periodicity.

The rough structure of the argument in \textcite{BloomSisask2019} is to consider, for a
subset $A$ of $\mathbf{Z}/N\mathbf{Z}$ lacking nontrivial three-term arithmetic
progressions, the $L^{p}$-norm of the convolution $1_A*1_A$ for large $p$ (on the order of
$\log(1/\alpha)$), and then to deduce a density-increment in both the case when
$\|1_A*1_A\|_{L^p}$ is small and the case when $\|1_A*1_A\|_{L^p}$ is
large. \textcite{BloomSisask2020} required a more flexible version of this second part of
their earlier argument, which we record below. Recall that $g\circ h\coloneqq g*h_{-}$, where
$h_{-}(x)\coloneqq \overline{h(-x)}$.
\begin{lemm}[\cite{BloomSisask2020}, Lemma~5.10]
  \label{lem:convoinc}
  Let $K\geq 10$ be a parameter, $B\subset\mathbf{Z}/N\mathbf{Z}$ be a regular Bohr set of
  rank $r$, $A\subset B$ have density $\alpha\leq 1/K$, $\rho\ll \alpha^2 r$, and
  $B'\subset B_\rho$ be another Bohr set of rank $r$. If
  \begin{equation*}
    \|\mu_A\circ 1_A\|_{L^{2m}(\mu_{B'}\circ\mu_{B'})}\geq \alpha K,
  \end{equation*}
  then $A$ has a density-increment relative to $B'$ of strength $[K,\frac{1}{\alpha K};\tilde{O}_\alpha(m\alpha^{-O(1/m)})]$.
\end{lemm}
Note that this result does not require $A$ to lack three-term arithmetic
progressions--that hypothesis is only used in \textcite{BloomSisask2019} when
$\|1_A*1_A\|_{L^{2m}}$ is small. The proof of Lemma~\ref{lem:convoinc} is
short, and utilizes an $L^p$-almost-periodicity result relative to Bohr sets. But the
proof is even shorter in the finite field model setting, and the relevant $L^p$-almost
periodicity result quicker to state, so we will instead present the model proof, which
also appears in \textcite[Section~3]{BloomSisask2019}.

\begin{theo}[\cite{BloomSisask2019}, Theorem~3.2]
  \label{thm:almostp}
Let $p\geq 2$, $0<\varepsilon<1$, and $A\subset\mathbf{F}_3^n$ have density~$\alpha$. Then there
exists a subspace $V\leq\mathbf{F}_3^n$ of codimension
$\lesssim_{\varepsilon,\alpha}p/\varepsilon^2$ such that
\begin{equation*}
  \|\mu_A*1_A*\mu_V-\mu_A*1_A\|_{L^p}\leq\varepsilon\|\mu_A*1_A\|_{L^{p/2}}^{1/2}+\varepsilon^2.
  \end{equation*}
\end{theo}
The following lemma is a finite field model analogue of Lemma~\ref{lem:convoinc} for
bounded $K$.
\begin{lemm}[\cite{BloomSisask2019}, Lemma~3.4]
  Let $A\subset\mathbf{F}_3^n$ have density~$\alpha$ and $m$ be a natural number. If
  $\|\mu_A*1_A\|_{L^{2m}}\geq 10\alpha$, then there exists a subspace of codimension
  $\lesssim_{\alpha}m/\alpha$ such that $A$ has density at least $5\alpha$ on some
  translate of $V$.
\end{lemm}
\begin{proof}
  Applying Theorem~\ref{thm:almostp} with $p=2m$ and $\varepsilon=\sqrt{\alpha}/100$, say,
  gives us a subspace~$V$ of codimension $\lesssim_{\alpha}m/\alpha$ for which
  \begin{equation*}
    \|\mu_A*1_A*\mu_V-\mu_A*1_A\|_{L^{2m}}\leq\frac{\sqrt{\alpha}}{100}\|\mu_A*1_A\|_{L^m}^{1/2}+\frac{\alpha}{10000}.
  \end{equation*}
  Thus, by the reverse triangle inequality,
  \begin{equation*}
    \|\mu_A*1_A*\mu_V\|_{L^{2m}}\geq\|\mu_A*1_A\|_{L^{2m}}-\left(\frac{\sqrt{\alpha}}{100}\|\mu_A*1_A\|_{L^m}^{1/2}+\frac{\alpha}{10000}\right).
  \end{equation*}
  By hypothesis, $\|\mu_A*1_A\|_{L^{2m}}\geq10\alpha$, and, since we are on a probability
  space, $\|\mu_A*1_A\|_{L^m}\leq\|\mu_A*1_A\|_{L^{2m}}$. Hence,
  $\|\mu_A*1_A*\mu_V\|_{L^{2m}}$ is easily at least $5\alpha$. To finish, note that, again
  because we are on a probability space,
  \begin{equation*}
    5\alpha\leq \|\mu_A*1_A*\mu_V\|_{L^{2m}}\leq \|\mu_A*1_A*\mu_V\|_{L^{\infty}}\leq\|1_A*\mu_V\|_{L^\infty},
  \end{equation*}
  since $\mu_A$ has mean $1$. So $\|1_A*\mu_V\|_{L^\infty}\geq 5\alpha$, which precisely
  means that $A$ has density at least $5\alpha$ on a coset of $V$.
\end{proof}

We end this subsection by stating the finite field model version of an
$L^\infty$-almost-periodicity result used by \textcite{BloomSisask2020} in the ``spectral
boosting'' phase of their argument. This will be relevant to our discussion in
Section~\ref{sec:BS}.
\begin{lemm}[\cite{SchoenSisask2016}, Theorem~3.2]
  \label{lem:Linfty}
  Let $0<\varepsilon<1/2$ and $S,M,L\subset\mathbf{F}_3^n$ where $S$ has density $\sigma$
  and $|M|/|L|=\nu$. There exists a subspace $W\leq \mathbf{F}_3^n$ of codimension at most
  $\lesssim _{\nu\sigma\varepsilon}\varepsilon^{-2}$ such that
  \begin{equation*}
    \|\mu_S*\mu_M*1_L*\mu_V-\mu_S*\mu_M*1_L\|_{L^\infty}\leq\varepsilon.
  \end{equation*}
\end{lemm}

\subsection{Higher energies of the large spectrum and additive non-smoothing}

In the course of their argument, \textcite{BatemanKatz2012} undertook a close study of the
additive and higher energies of large spectra of cap-sets. Let $A$ be a subset of an
abelian group $G$. The \emph{additive energy} $E_4(A)$ of $A$, a central notion in
additive combinatorics, is defined as the number of \emph{additive quadruples} in $A$,
\begin{equation*}
  E_4(A)\coloneqq \left|\{(a_1,a_2,a_3,a_4)\in A^4\mid a_1+a_2=a_3+a_4\}\right|.
\end{equation*}
Note the trivial upper and lower bounds $|A|^3\geq E_4(A)\geq |A|^2$. If $A$~is a subgroup
of~$G$, then $E_4(A)=|A|^3$ is maximal, and if, more generally, $A$~is a coset progression
of~$G$ of bounded rank, then $E_4(A)\asymp |A|^3$. At the opposite extreme, if $A$~is a
random subset of~$G$, then $E_4(A)$ is close to the minimum~$|A|^2$.

Additive energy is a convenient measure of the degree to which a set possesses additive
structure, and can be translated into other notions of additive structure, often with only
polynomial losses. For example, the Balog--Szemer\'edi--Gowers theorem says that sets with
large additive energy must contain a large subset with small doubling. Additive energy is
a particularly nice measure of additive structure to work with because it has a simple
expression in terms of the inverse Fourier transform,
\begin{equation*}
  E_4(A)=\mathbf{E}_{x\in G}|\widecheck{1_A}(x)|^4,
\end{equation*}
so that it can be manipulated using analytic methods.

There are also higher energies whose study has been useful in additive combinatorics. For
every natural number $m$, we define
\begin{equation*}
  E_{2m}(A)\coloneqq \left|\left\{(a_1,a_1',\dots,a_m,a_m')\in
    A^{2m}\mid \sum_{i=1}^ma_i=\sum_{i=1}^ma_i'\right\}\right|=\mathbf{E}_{x\in G}|\widecheck{1_A}(x)|^{2m}.
\end{equation*}
Note the trivial upper and lower bounds $|A|^{2m-1}\geq E_{2m}(A)\geq |A|^{m}$. By
H\"older's inequality $E_4(A)^{m-1}\leq E_{2m}(A)|A|^{m-2}$ for all $m>2$ and, similarly,
$E_8(A)^{\frac{m-1}{3}}\leq E_{2m}(A)|A|^{\frac{m-4}{3}}$ for all $m>4$. If we set $\tau$
to be the normalized additive energy $\tau\coloneqq  E_4(A)/|A|^3$, then
$E_{2m}(A)\geq \tau^{m-1}|A|^{2m-1}$ for all $m> 2$, and if we set $\sigma$ to be the
normalized higher energy $\sigma\coloneqq  E_8(A)/|A|^7$, then
$E_{2m}(A)\geq \sigma^{\frac{m-1}{3}}|A|^{2m-1}$ for all $m>4$. Thus, if $E_4(A)$ or
$E_8(A)$ is large, then so are the higher energies of $A$.

Chang's lemma says that large spectra have small dimension, which is one sense in
which they are additively structured. Large spectra also have decently large additive
energy. Indeed, writing $z(\xi)=\widehat{1_A}(\xi)/|\widehat{1_A}(\xi)|$ for each
$\xi\in\Spec_{\delta}(A)$ with $\widehat{1_A}(\xi)\neq 0$ and inserting the Fourier inversion formula for $1_A$, we have
\begin{align*}
  \alpha\delta\lvert\Spec_\delta(A)\rvert\leq&\sum_{\xi}\widehat{1_A}(\xi)z(\xi)1_{\Spec_\delta(A)}(\xi)\\
  =&\mathbf{E}_{x}1_A(x)\left(\sum_{\xi}z(\xi)e_3(-\xi\cdot
     x)1_{\Spec_\delta(A)}(\xi)\right)\\
  &\leq \alpha^{(2m-1)/2m}E_{2m}(\Spec_{\delta}(A))^{1/2m},
\end{align*}
by applying H\"older's inequality with exponents $2m$ and $\frac{2m}{2m-1}$, from which it
follows that
\begin{equation*}E_{2m}(\Spec_{\delta}(A))\geq\alpha\delta^{2m}\lvert\Spec_\delta(A)\rvert^{2m}.\end{equation*}

The first key insight of Bateman and Katz is that sets with a large higher energy contain
a positive density subset of small dimension. An instance of this relative to Bohr sets
was proven by \textcite{Bloom2016}, and combined with (a more technical version of) the
observation that large spectra have large higher energies to prove his alternative to
Chang's lemma. Bloom worked with relativized notions of additive and higher energies, and
obtained a conclusion involving relativized notions of dissociativity and
dimension. \textcite{BloomSisask2020} also required a variant of Bloom's result, and the
following lemma is a special case of their Lemma~7.9.
\begin{lemm}
  \label{lem:simplifiedenergy}
  There exists an absolute constant $C>0$ such that the following holds. Let
  $\Delta\subset\mathbf{Z}/N\mathbf{Z}$ and $\ell,m\geq 2$ be integers satisfying
  $\ell\geq 4m$. Then either
  \begin{enumerate}
  \item there exists a subset $\Delta'\subset\Delta$ such that
    \begin{equation*}
      |\Delta'|\geq\min\left(1,\frac{|\Delta|}{\ell}\right)\frac{m}{2\ell}|\Delta|
    \end{equation*}
    and $\dim{\Delta'}\ll\ell$, or
  \item $E_{2m}(\Delta)\leq (Cm/\ell)^{2m}|\Delta|^{2m}$.
  \end{enumerate}
\end{lemm}

When $A$~is a cap-set, by dyadic pigeonholing we can find a $1\geq\delta\gg\alpha$ for
which $\lvert\Spec_{\delta}(A)\rvert\gtrsim_{\alpha}\delta^{-3}$ and~\eqref{eq:largel2} holds. If
$\delta$~is substantially larger than $\alpha$, say $\delta>K^2\alpha$ for~$K$ a very
small power of~$\alpha^{-1}$, then the (finite field model version of) Lemma~\ref{lem:L2Bohr}
combined with the lower bound
$E_{2m}(\Spec_{\delta}(A))\geq\alpha\delta^{2m}\lvert\Spec_\delta(A)\rvert^{2m}$ and repeated
applications of (the finite field model version of) Lemma~\ref{lem:simplifiedenergy}
produces a density-increment of strength $[K,1/\alpha K;1]$ or $[1/K,1;1]$, both of which
would be good enough to obtain the bound $r_3(\mathbf{F}_3^n)\ll\frac{3^n}{n^{1+c}}$ for
some small $c$. So, now suppose that
$K^2\alpha\geq\delta\gg\alpha$. In this case we must also have
$1/\delta^3\lesssim_\alpha\lvert\Spec_\delta(A)\rvert\leq K^2/\delta^3$, and if one of
$E_4(\Spec_\delta(A))$ or $E_8(\Spec_\delta(A))$ is substantially larger than their
minimal values of $\delta^{-7}$ and $\delta^{-15}$, respectively, say
$E_4(\Spec_\delta(A))\geq L\delta^{-7}$ or $E_8(\Spec_\delta(A))\geq L\delta^{-15}$ for~$L$ equal to another small power of~$\alpha^{-1}$, then, by our earlier discussion, the higher
energies of $\Spec_\delta(A)$ must be large, so that we can again obtain a good enough
density-increment by again combining Lemma~\ref{lem:L2Bohr} with repeated applications of
Lemma~\ref{lem:simplifiedenergy}.

The only remaining case to handle in the proof of Bateman and Katz is when
\begin{equation}
  \label{eq:addnonsmooth}
  \delta\approx \alpha,\ \lvert\Spec_\delta(A)\rvert\approx\delta^{-3},\ 
  \frac{E_4(\Spec_\delta(A))}{\lvert\Spec_\delta(A)\rvert^3}\approx\delta^2,\ \text{and}\  \frac{E_8(\Spec_\delta(A))}{\lvert\Spec_\delta(A)\rvert^7}\approx \delta^{6},
\end{equation}
where we will temporarily use $\approx$ to hide small powers of $\alpha^{-1}$. Recall that if
$\tau$ is the normalized additive energy and $\sigma$ is the normalized $E_8$-energy, then
$\sigma\geq\tau^3$. Thus, $E_8(\Spec_\delta(A))$ is about as small as it can be given the
size of $E_4(\Spec_\delta(A))$. \textcite{BatemanKatz2012} call sets with this property
\emph{additively non-smoothing} and all other sets \emph{additively
  smoothing}. Additive energy measures the additive structure of $A$, and the $E_8$-energy
similarly measures the additive structure of $A+A$. Thus, a set is additively smoothing if
its sumset is substantially more structured than itself, and additively non-smoothing if
forming the sumset does not improve the additive structure. For example, a random subset
of $\mathbf{F}_3^n$ is additively smoothing, while an affine subspace is additively
non-smoothing. Two slightly more elaborate examples of additively non-smoothing sets,
highlighted in \textcite{BatemanKatz2011}, are, for parameters $M\geq 1$ and
$1>\gamma>0$, sets of the form $H+R$ where $H$ is a subgroup of order $M^{1-\gamma}$ and
$R$ is a random set of size $M^{\gamma}$, and unions of $M^{\gamma/2}$ random subspaces of
order $M^{1-\gamma/2}$.

The second key insight of Bateman and Katz is that it is possible to classify additively
non-smoothing sets, and that such a classification could be used to deduce a strong
density-increment. They proved a structure theorem saying, roughly, that if
$S\subset\mathbf{F}_3^n$ is additively non-smoothing, then a large portion of $S$ can be
decomposed into a union of sumsets of the form $X+H$, where $H$ is very additively
structured. By applying this result to $S=\Spec_\delta(A)$, they eventually managed to
obtain a strong density-increment.

\textcite{BloomSisask2020}, too, needed a structure theorem for additively non-smoothing
sets, now also relative to ``additive frameworks'' of large spectra of Bohr sets, and
using relative notions of additive and higher energies. Proving such a result is the most
difficult and complex part of their argument, and required them to come up with a more
robust proof in the finite field model setting that could be relativized. We will not give
the definition of an additive framework, nor the precise definition of an additively
non-smoothing set (relative to an additive framework). The following is a simplified
portrayal of the structure theorem of \textcite[Theorem~9.2]{BloomSisask2020}.
\begin{roughtheo}
  Let $\tau\leq 1/2$ be a parameter, $G$ be a finite abelian group, and $\tilde{\Gamma}$
  be a suitable additive framework. If $E_4(\Delta)/|\Delta|^3=\tau$ and $\Delta$ is
  non-smoothing relative to $\tilde{\Gamma}$, then there exist subsets $X,H\subset\Delta$
  and $1\geq\gamma\gg\tau$ such that
  \begin{equation}
    \label{eq:Structure1}
    |H|\asymp\gamma|\Delta|\qquad\text{and}\qquad|X|\asymp\frac{\tau}{\gamma}|\Delta|,
  \end{equation}
  and
  \begin{equation}
    \label{eq:Structure2}
    \langle 1_X\circ 1_X,1_H\circ 1_H\circ 1_{\Gamma_{\text{top}}}\rangle\gg|H|^2|X|.
  \end{equation}
\end{roughtheo}
In the finite field model setting, $\tilde{\Gamma}$ can be taken to be trivial, so that
the condition~\eqref{eq:Structure2} becomes
$\langle 1_X\circ 1_X,1_H\circ 1_H\rangle\gg|H|^2|X|$.  If desired, one can derive a
structure theorem of the form of that of Bateman and Katz by iterating this lemma and
applying the asymmetric Balog--Szemer\'edi--Gowers theorem.

\section{The argument of Bloom and Sisask}\label{sec:BS}

The argument of \textcite{BloomSisask2020} broadly follows the path of Bateman and Katz,
but working relative to Bohr sets instead of subspaces. Bloom and Sisask had to overcome
multiple significant obstacles in the integer setting that were not present in the finite
field model setting, several of which we have already mentioned. One obstacle not yet
mentioned (because we did not give any details on the proof in the finite field model
setting) is that one part of the argument that Bateman and Katz used to go from their
structure theorem for additively non-smoothing sets to a strong density-increment does not
have an efficient analogue in the integer setting.

Bloom and Sisask, like Bateman and Katz, can produce density-increments sufficiently large
to prove Theorem~\ref{thm:bs} through the methods discussed in
Section~\ref{sec:Background}, unless~\eqref{eq:largel2} and~\eqref{eq:addnonsmooth} hold
for some $\alpha K\gg\delta\gg\alpha$, where $K$ is a small power of $\alpha^{-1}$. This means
that $\Spec_\delta(A)$ is additively non-smoothing, and Bloom and Sisask can then
iteratively apply their structure theorem to decompose a significant portion of
$\Spec_\delta(A)$ into a union of structured sets, from which they can deduce a
density-increment provided that the structured pieces are all, individually,
sufficiently large, say of size $\Omega(L/\alpha)$ for $L$ some other small power of
$\alpha^{-1}$. The final remaining case is thus when~\eqref{eq:largel2}
and~\eqref{eq:addnonsmooth} hold for some $\alpha K\gg\delta\gg\alpha$ and the structure
theorem for additively non-smoothing sets produces an $H$ of size at most
$L/\alpha$. Bloom and Sisask derive a strong density-increment in this situation via a new
argument that they call ``spectral boosting''.

Bloom and Sisask call this last piece of their proof ``spectral boosting'' because they
obtain a density-increment of the strength one would obtain if the structured set $H$ were
contained in the $\sqrt{\alpha}$-large spectrum, instead of the $\alpha$-large
spectrum. Thus, the elements $H$ can be viewed as morally ``boosted'' to a larger
spectrum. To finish off this section, we will give a sketch of the spectral boosting
argument in the finite field model setting.

Suppose, for the sake of illustration, that $\delta=\alpha$,
$\lvert\Spec_\alpha(A)\rvert\asymp 1/\alpha^3$, $H\subset \Spec_\alpha(A)$ satisfies
$|H|\asymp 1/\alpha$ and $\dim{H}\lesssim_\alpha 1$, and $X\subset \Spec_\alpha(A)$
satisfies $|X|\asymp 1/\alpha^3$ and $E_4(X,H)\gg |X||H|^2$.  We may remove $0$ from $X$
without affecting the lower bound on the relative energy by much, so set
$X'\coloneqq \Spec_{\alpha}(A)\setminus\{0\}$, so that
\begin{equation*}
 \sum_{\xi_1+\xi_2=\xi_3+\xi_4}1_H(\xi_1)1_{X'}(\xi_2)1_H(\xi_3)1_{X'}(\xi_4)\gg|X'||H|^2.
\end{equation*}
Consider the function $f\coloneqq 1_A*1_A-\alpha^2$, which has Fourier transform $\widehat{f}$
equal to $|\widehat{1_A}|^2$ on $X'$. Since $|\widehat{1_A}|\geq\alpha^2$ on $X'$ by
definition, $\widehat{f}\geq \alpha^4 1_X$, so that we can replace the first instance of
$1_{X'}$ with $\alpha^{-4}\widehat{f}$ to obtain
\begin{equation*}
 \sum_{\xi_1+\xi_2=\xi_3+\xi_4}1_H(\xi_1)\widehat{f}(\xi_2)1_H(\xi_3)1_{X'}(\xi_4)\gg\alpha^4|X'||H|^2.
\end{equation*}
Taking inverse Fourier transforms then gives
\begin{equation*}
  \alpha^4|X'||H|^2\ll\mathbf{E}_x|\widecheck{1_{H}}(x)|^2f(x)\widecheck{1_{X'}}(x).
\end{equation*}
One important thing to note here is that $|\widecheck{1_H}|$ is invariant under shifts by
elements that annihilate $H$. These elements form a subspace $V\leq\mathbf{F}_3^n$ of
codimension at most the dimension of $H$. We would like to remove $\widecheck{1_{X'}}$
from the average on the right so that we can obtain a large correlation of $|f|$ with
$|\widecheck{1_H}|^2$, which will allow us to freely convolve with $1_V$ later in the
argument and eventually obtain a density-increment on a translate of a subspace of $V$ of
small codimension.

The easiest way to remove $\widecheck{1_{X'}}$ is to apply H\"older's inequality with
$p=1$ and $q=\infty$ to obtain
$\|\widecheck{1_{X'}}\|_{L^\infty}\cdot\mathbf{E}_x|\widecheck{1_{H}}(x)|^2|f(x)|$. If
$X'$ is not additively structured, then its inverse Fourier transform $\widecheck{1_{X'}}$
should, morally, behave like a random sum of characters, and thus be small, making this an
efficient use of H\"older's inequality. On the other hand, if $X'$~is additively
structured, we can already obtain a strong density-increment since $X'$~has positive
density in~$\Spec_\alpha(A)$.

To make the above intuition rigorous, we will apply H\"older's inequality with $q=m$ and
$p=\frac{m}{m-1}$ with large $m$, and then use the Cauchy--Schwarz inequality, yielding
\begin{align*}
  \mathbf{E}_x|\widecheck{1_{H}}(x)|^2f(x)\widecheck{1_X}(x) &=
                                                               \mathbf{E}_x(|\widecheck{1_{H}}(x)|^2f(x))^{1-1/m}\cdot(|\widecheck{1_{H}}(x)|^2f(x))^{1/m}\widecheck{1_{X'}}(x)
  \\
  &\leq
    \||\widecheck{1_{H}}|^2f\|_{L^{1}}^{\frac{m-1}{m}}\left(\mathbf{E}_{x}|\widecheck{1_{H}}(x)|^2|f(x)||\widecheck{1_{X'}}(x)|^{m}\right)^{1/m}
  \\
  &\leq
    \||\widecheck{1_{H}}|^2f\|_{L^{1}}^{\frac{m-1}{m}}\||\widecheck{1_{H}}|^2f\|_{L^{2}}^{\frac{1}{m}}E_{2m}(X')^{1/2m}.
\end{align*}
Parseval's identity and Cauchy--Schwarz give us
$\||\widecheck{1_{H}}|^2f\|_{L^{2}}^2\leq|H|^4/\alpha$, from which it follows that
\begin{equation*}
\alpha^{8m+1}|X'|^{2m}|H|^{4m-4}\ll \||\widecheck{1_{H}}|^2f\|_{L^{1}}^{2m-2}\cdot E_{2m}(X').
\end{equation*}
Now, we take $m\lesssim_{\alpha}1$, so that if the higher energy $E_{2m}(X')$ is large,
say $\gg (\alpha L|X'|)^{2m}$ for $L=\alpha^{-1/1000}$, then we can obtain a
density-increment of strength $[\alpha^{-1/1000},\alpha^{-999/1000};O(1)]$ as previously
discussed. We may therefore assume that $E_{2m}(X')\ll(\alpha L|X'|)^{2m}$, which yields
\begin{equation}
  \label{eq:sb1}
\frac{\alpha^2}{L}|H|\asymp\alpha^{O(1/m)}\alpha^{3}|H|^{2}\ll \mathbf{E}_x|\widecheck{1_{H}}(x)|^2|f(x)|.
\end{equation}

Note that if~\eqref{eq:sb1} held with $f$ in place of $|f|$, then
\begin{equation*}
  \frac{\alpha}{L}|H|\ll\sum_{\xi}1_{H}\circ 1_{H}(\xi)|\widehat{f_A}(\xi)|^2
\end{equation*}
by Parseval's identity, so that, by the pigeonhole principle, there exists a translate
$z+H$ of $H$ for which
\begin{equation*}
  \frac{\alpha}{L}\ll\sum_{\xi\in z+H}|\widehat{f_A}(\xi)|^2,
\end{equation*}
thus producing a very large density-increment. We do indeed have to deal
with $|f|$, however.

Observe that if $T\coloneqq \{x\mid f(x)\geq c\alpha^2\}$ and $T'\subset T$, then
\begin{equation*}
 \mathbf{E}_{x}1_A*1_A(x)1_{T'}(x)=\mathbf{E}_xf(x)1_{T'}(x)+\mathbf{E}_{x}\alpha^21_{T'}(x)\geq (1+c)\alpha^2\mu(T'),
\end{equation*}
which looks encouragingly similar to a strong density-increment. If we can show that $T$
has density $\frac{1}{K}\asymp\alpha^{1/500}$ on a translate of $u+V$, then we can take
$T'\coloneqq T\cap (u+V)$, from which it follows that there exists a translate $A'$ of $A$ and a
subset $S\subset V$ of density $\gg \frac{1}{K}$ in $V$ for which
\begin{equation*}
  \mathbf{E}_{x}1_S*1_{A'}(x)1_{A}(x)\geq (1+c)\alpha^2\mu(S).
\end{equation*}
By splitting up $1_A$ into a sum of indicator functions of $A$ intersected with cosets of
$V$ and applying the pigeonhole principle, we may replace $A'$ and $A$ with intersections
$A''$ and $A'''$ of $A$ with cosets of $V$, yielding an inequality of the form
\begin{equation*}
  \mathbf{E}_{x}1_S*1_{A''}(x)1_{A'''}(x)\geq (1+c)\alpha\mu(A'')\mu(S).
\end{equation*}
The expression on the left-hand side is a convolution $1_S*1_{A''}\circ 1_{A'''}$, which
can be approximated by applying Lemma~\ref{lem:Linfty} relative to $V$ with
$\varepsilon\asymp 1$ to find a subspace $W\leq V$ of codimension $\lesssim_{\alpha}1$ in
$V$ for which
\begin{equation*}
  \mathbf{E}_{x}1_S*1_{A''}(x)1_{A'''}*\mu_V(x)\geq (1+c)\alpha\mu(A'')\mu(S).
\end{equation*}
The existence of a density-increment of strength $[1,\tilde{O}_\alpha(1);O(1)]$ now follows
by applying H\"older's inequality with $p=1$ and $q=\infty$ and noting that
$\|1_S*1_{A''}\|_{L^1}=\mu(A'')\mu(S)$.

It thus remains to show that $T$ has density $\asymp \alpha^{1/500}$ on a translate of
$u+V$. The first step is to remove the absolute value bars around $f$ in~\eqref{eq:sb1} by
restricting to a subset on which $f$ is large and positive. Using again the identity
$r+|r|=2\max(r,0)$, we have
\begin{equation*}
  \frac{\alpha^2}{L}|H|\ll \mathbf{E}_x|\widecheck{1_{H}}(x)|^2\max(f(x),0).
\end{equation*}
Letting $T\coloneqq \{x\mid f(x)\geq c\alpha^2\}$ for some suitably small absolute constant $c$, it follows that
\begin{equation*}
  \frac{\alpha^2}{L}|H|\ll \mathbf{E}_x|\widecheck{1_{H}}(x)|^21_T(x)f(x).
\end{equation*}
We will remove $f$ from the average by applying H\"older's inequality with exponents
$p=2m$ and $q=\frac{2m}{2m-1}$ for $m\asymp\log(1/\alpha)$ and then the bound
$|\widecheck{1_{H}}|\leq|H|$ to obtain
\begin{align*}
  \mathbf{E}_x|\widecheck{1_{H}}(x)|^21_T(x)f(x)&\leq
  \|f\|_{L^{2m}}\left(\mathbf{E}_x|\widecheck{1_{H}}(x)|^{2+2/(2m-1)}1_T(x)\right)^{1-1/2m}
  \\
  &\leq \|f\|_{L^{2m}}\left(\frac{|H|}{\mathbf{E}_x|\widecheck{1_{H}}(x)|^{2}1_T(x)}\right)^{1/m}\left(\mathbf{E}_x|\widecheck{1_{H}}(x)|^{2}1_T(x)\right).
\end{align*}
Recalling the definition of $f$, by the triangle inequality,
$\|f\|_{L^{2m}}\leq \alpha^2+\|1_A*1_A\|_{L^{2m}}$. If $\|1_A*1_A\|_{L^{2m}}$ were
$\gg\alpha^{2-1/1000}$, then we would be able to obtain a density-increment of strength
$[\alpha^{-1/1000},\alpha^{-999/1000};O(1)]$ using (a finite field model version of)
Lemma~\ref{lem:convoinc}, which is certainly good enough. We may therefore proceed under
the assumption that $\|f\|_{L^{2m}}\ll \alpha^{2-1/1000}$. Since
$\mathbf{E}_x|\widecheck{1_{H}}(x)|^{2}1_T(x)\geq\alpha^2|H|$, the second term above is
$\ll \alpha^{-2/m}$. It therefore follows that
\begin{equation*}
  \frac{1}{L^2}|H|\ll\mathbf{E}_x|\widecheck{1_{H}}(x)|^{2}1_T(x)=\mathbf{E}_x|\widecheck{1_{H}}(x)|^{2}1_T*\mu_V(x).
\end{equation*}
Finally, by the pigeonhole principle, there must exist an $x$ for which
$1_T*\mu_V(x)\gg \frac{1}{L^2}$, i.e., $T$ has density $\gg\alpha^{1/500}$ on $x+V$.

\section{The Croot--Lev--Pach polynomial method and the work of Ellenberg--Gijswijt}\label{sec:EG}

Many proof techniques that fall under the umbrella of the polynomial method tend to follow
the same rough structure:
\begin{itemize}
\item First, the key data of the object of study is encoded in one or several polynomials,
  typically of low ``complexity''.
\item Then, the algebraic properties of low complexity polynomials are used to obtain the
  desired conclusion.
\end{itemize}
The Croot--Lev--Pach polynomial method also follows this outline. In this final section,
we will present a full proof of power-saving bounds in the cap-set problem using the
``slice rank'' method of \textcite{Tao2016}, which is a symmetric rephrasing of the
argument of \textcite{EllenbergGijswijt2017}.

To define slice rank, we first specify the functions of slice rank one.
\begin{defi}
  Let $S$ be a finite set, $k$ be a positive integer, and $\mathbf{F}$ be a field. We say
  that a function $f\colon S^k\to\mathbf{F}$ has \emph{slice rank one} if there exists an
  index $1\leq i\leq k$, a function $g\colon S\to\mathbf{F}$, and a function
  $h\colon S^{k-1}\to\mathbf{F}$ such that
  \begin{equation*}
    f(x_1,\dots,x_k)=g(x_i)h(x_1,\dots,x_{i-1},x_{i+1},\dots,x_k).
  \end{equation*}
\end{defi}
Thus, a function of $k$ variables on $S$ has slice rank one if it can be written as a
product of a function of one variable and a function of the remaining $k-1$
variables. Having defined the functions of slice rank one, the slice rank can now be
defined, analogously to other notions of rank such as tensor rank, as the minimum number
of rank one functions needed to represent a given function.
\begin{defi}
  Let $S$ be a finite set, $k$ be a positive integer, and $\mathbf{F}$ be a field. We say
  that a function $f\colon S^k\to\mathbf{F}$ has \emph{slice rank} at most $m$ if there exist
  $m$ functions $g_1,\dots,g_m\colon S^k\to\mathbf{F}$ of slice rank one such that
  \begin{equation*}
    f(x_1,\dots,x_k)=\sum_{j=1}^mg_j(x_1,\dots,x_k).
  \end{equation*}
  The \emph{slice rank} of $f$ is the smallest $m_0$ for which $f$ has slice rank at
  most $m_0$.
\end{defi}

Now suppose that $A\subset\mathbf{F}_3^n$ is a cap-set, and let $f\colon A\times A\times A\to\mathbf{F}_3$ denote the indicator function
 of the diagonal of $A\times A\times A$,
\begin{equation}
  \label{eq:samethree}
  f(x,y,z)\coloneqq 
  \begin{cases}
    1 & x=y=z \\
    0 &\text{otherwise.}
  \end{cases}
\end{equation}
The idea of the proof of Theorem~\ref{thm:eg} is to
\begin{itemize}
\item show that $f$ has slice rank exactly $|A|$,
\item express $f$ as an nice, explicit polynomial,
\item and then bound the slice rank of this polynomial,
\end{itemize}
thus producing a bound for the cardinality $|A|$. The assumption that $A$ is a cap-set is
only used in the second step of this outline, where it is, obviously, crucial.

We begin by proving that the slice rank of $f$ is $|A|$.
\begin{lemm}
  \label{lem:srequal}
  Let $S$ be a finite subset and $\mathbf{F}$ be a field, and define
  $f\colon S\times S\times S\to\mathbf{F}$ by~\eqref{eq:samethree}. Then the slice rank of $f$
  is $|S|$.
\end{lemm}
\begin{proof}
  Since
  \begin{equation*}
    f(x,y,z)=\sum_{s\in S}1_{\{s\}}(x)1_{\{s\}}(y)1_{\{s\}}(z),
  \end{equation*}
  the function $f$ certainly has slice rank at most $|S|$. To show that the slice rank is
  at least $|S|$, suppose by way of contradiction that the slice rank equals some positive
  integer $t<|S|$. Then we may assume, without loss of generality, that there are
  functions $g_1,\dots,g_t\colon S\to\mathbf{F}$ and $h_1,\dots,h_t\colon S\times S\to\mathbf{F}$ such
  that
  \begin{equation*}
    f(x,y,z)=\sum_{i=1}^{t_1}g_i(x)h_i(y,z)+\sum_{i=t_1+1}^{t_2}g_i(y)h_i(x,z)+\sum_{i=t_2+1}^{t}g_i(z)h_i(x,y)
  \end{equation*}
  for some nonnegative integers $0\leq t_1\leq t_2< t$.

  The idea is now to find a function $r\colon S\to\mathbf{F}$ whose support
  $\supp{r}\coloneqq\{z\in S: r(z)\neq 0\}$ has size larger than $t_2$ for which
  $\sum_{z\in S}r(z)g_i(z)=0$ for all $i=t_2+1,\dots,t$, so that multiplying both sides of
  the above by $r(z)$ and summing over $z$ yields
  \begin{equation*}
    F(x,y) = \sum_{i=1}^{t_1}g'_i(x)h'_i(y)+\sum_{i=t_1+1}^{t_2}g'_i(x)h'_i(y),
  \end{equation*}
  where
  \begin{equation*}
    F(x,y)\coloneqq 
    \begin{cases}
      r(x) & x=y \\
      0 & \text{otherwise}
    \end{cases}
  \end{equation*}
  and $g_1',\dots,g_{t_2}',h_1',\dots,h_{t_2}'\colon S\to\mathbf{F}$.  In other words, the
  $|S|\times|S|$ diagonal matrix $D$ with entries $(r(s))_{s\in S}$ along the diagonal has
  rank at most $t_2$. But the rank of $D$ equals the number of nonzero elements along its
  diagonal, $\lvert\supp{r}\rvert$, which is greater than $t_2$, giving us a contradiction.

  To show that such a function $r$ exists, set $t'\coloneqq t-t_2$, and let $V$ denote the
  vector space over $\mathbf{F}$ of functions $S\to\mathbf{F}$ orthogonal to
  $g_{t_2+1},\dots,g_{t}$, so that $\dim V\geq|S|-t'$. Since $|S|-t'\geq |S|-t>0$,
  certainly $V$ contains some function that is not identically zero. Let $r\in V$ be a
  function with maximal support. If $\lvert\supp{r}\rvert<|S|-t'$, then the subspace of
  functions in $V$ that vanish on $\supp{r}$ has dimension at least one, and so contains
  some nonzero function $r'$. But then $r+r'$ would have strictly larger support than $r$,
  which contradicts $r$ having maximal support. So we must have
  $\lvert\supp{r}\rvert=|S|-t'$, i.e., $\lvert \supp{r}\rvert=t_2+|S|-t>t_2$.
\end{proof}

Next, we will express $f$ as a low-complexity polynomial, and derive an upper bound for
its slice rank.
\begin{lemm}
  \label{lem:srbound}
  Let $A\subset\mathbf{F}_3^n$ be a cap-set and $f\colon A\times A\times A\to\mathbf{F}_3^n$ be
  defined as in~\eqref{eq:samethree}. Then the slice rank of $f$ is at most
  \begin{equation}
    \label{eq:srbound}
    M:=3\cdot\left|\left\{(a_1,\dots,a_n)\in\{0,1,2\}^n\mid \sum_{i=1}^na_i<\frac{2n}{3}\right\}\right|.
\end{equation} 
\end{lemm}
\begin{proof}
  Since $A$ is a cap-set, the only solutions to the equation $x+y+z=0$ with $x$, $y$, and
  $z$ all in $A$ are the trivial solutions $x=y=z$. This means that
  \begin{equation*}
    f(x,y,z)=1_{\{0\}}(x+y+z).
  \end{equation*}
  Note that for any element $w\in\mathbf{F}_3$, we have
  \begin{equation*}
    1-w^2=
    \begin{cases}
      1 & w=0 \\
      0 & w\neq 0.
    \end{cases}
  \end{equation*}
  Thus,
  \begin{equation*}
    1_{\{0\}}(x+y+z)=\prod_{i=1}^n(1-(x_i+y_i+z_i)^2)\eqqcolon P(x,y,z).
  \end{equation*}
  The polynomial $P$ has degree $2n$, and every monomial appearing in $P$ takes the form
  \begin{equation*}
    x_1^{a_1}\cdots x_n^{a_n}y_1^{b_1}\dots y_n^{b_n}z_1^{c_1}\cdots z_n^{c_n},
  \end{equation*}
  where $0\leq a_i,b_j,c_k\leq 2$ for each $1\leq i,j,k\leq n$ and
  $\sum_{i=1}^na_i+\sum_{j=1}^nb_j+\sum_{k=1}^nc_k\leq 2n$. For each such monomial, one of
  $\sum_{i=1}^na_i$, $\sum_{j=1}^nb_j$, or $\sum_{k=1}^nc_k$ is less than $2n/3$. Thus,
  $P$ can be written as
  \begin{align*}
    P(x,y,z)=&\sum_{\substack{0\leq a_1,\dots,a_n\leq 2 \\
    a_1+\dots+a_n<2n/3}}x_1^{a_1}\cdots x_n^{a_n}g_{\mathbf{a}}(y,z)+\sum_{\substack{0\leq
    b_1,\dots,b_n\leq 2 \\ b_1+\dots+b_n<2n/3}}y_1^{b_1}\cdots
    y_n^{b_n}h_{\mathbf{b}}(x,z) \\
    &+\sum_{\substack{0\leq c_1,\dots,c_n\leq 2 \\ c_1+\dots+c_n<2n/3}}z_1^{c_1}\cdots z_n^{c_n}r_{\mathbf{c}}(x,y)
  \end{align*}
  for some functions
  $g_{\mathbf{a}},h_{\mathbf{b}},r_{\mathbf{c}}\colon S\times S\to\mathbf{F}$. It therefore follows that
  the slice rank of $1_{\{0\}}(x+y+z)$ is at most~\eqref{eq:srbound}.
\end{proof}

Now consider a sequence of $n$ random variables $X_1,\dots,X_n$ taking values
independently and uniformly in $\{0,1,2\}$. The probability that
$X\coloneqq X_1+\dots+X_n$ is smaller than $2n/3$ equals $M$ times $1/3^{n+1}$, and this
probability is at most $2e^{-n/18}$ by Hoeffding's inequality. Hence, the slice rank of
$f$ is $\ll(3/e^{1/18})^n$ by Lemma~\ref{lem:srbound}, and so
$|A|\ll(3/e^{1/18})^n\approx 2.838^n$ by Lemma~\ref{lem:srequal}. Obtaining the bound of
$\ll 2.756^n$ appearing in the theorem of Ellenberg and Gijswijt just requires a less
crude estimation of $M$, and is straightforward, though a bit tedious.

\printbibliography

\end{document}

%%% Local Variables:
%%% mode: latex
%%% TeX-master: t
%%% End: